\documentclass[a4paper]{article}
\usepackage{amssymb,amsmath,amsfonts,mathptmx}
\usepackage{graphicx}
 \graphicspath{{./pics/},{./picsLemma1/}}
 \DeclareGraphicsExtensions{.pdf}
\usepackage{stmaryrd}
\usepackage{authblk}
\usepackage{amsthm}
\usepackage{geometry}
\usepackage{tikz}
  \usetikzlibrary{arrows,automata,shapes}
 
\newtheorem{theorem}{Theorem}
\newtheorem{lemma}[theorem]{Lemma}
\newtheorem{proposition}[theorem]{Proposition}
\newtheorem{assumption}[theorem]{Assumption}

\theoremstyle{definition}

\newtheorem{example}{Example}

\newcommand{\eps}{\varepsilon}
\newcommand{\supc}{\mbox{$\sup {\rm C}$}}
\newcommand{\supcn}{\mbox{$\sup {\rm CN}$}}

\newcommand{\R}{\mathbf{R}}

\providecommand{\keywords}[1]{\textbf{\textit{Index terms---}} #1}

\begin{document}

\title{Computation of Controllable and Coobservable Sublanguages in Decentralized Supervisory Control via Communication}

\author[*]{Jan Komenda}
\author[**]{Tom{\' a}{\v s}~Masopust}

\affil[*]{Institute of Mathematics, Czech Academy of Sciences, {\v Z}i{\v z}kova 22, 616 62 Brno, Czech Republic, {\tt komenda{@}ipm.cz}}

\affil[**]{Institute of Theoretical Computer Science and Center of Advancing Electronics Dresden (cfaed), TU Dresden, Germany, {\tt tomas.masopust{@}tu-dresden.de} }

\date{}

\maketitle

\begin{abstract}
  In decentralized supervisory control, several local supervisors cooperate to accomplish a common goal (specification). Controllability and coobservability are the key conditions to achieve a specification in the controlled system. We construct a controllable and coobservable sublanguage of the specification by using additional communications between supervisors. Namely, we extend observable events of local supervisors via communication and apply a fully decentralized computation of local supervisors. Coobservability is then guaranteed by construction. Sufficient conditions to achieve the centralized optimal solution are discussed. Our approach can be used for both prefix-closed and non-prefix-closed specifications.
  
  \keywords{Discrete-event systems \and Decentralized supervisory control \and Coobservability \and Separability \and Communication}
\end{abstract}

\section{Introduction}
  Supervisory control theory of discrete-event systems (DES) modeled by finite automata was introduced by Ramadge and Wonham~\cite{RW87}. It aims to guarantee that the control specification consisting of safety and/or nonblockingness is satisfied in the controlled system. Supervisory control is realized by a supervisor that runs in parallel with the system and imposes the specification by disabling some of the controllable events in a feedback manner.

  Decentralized supervisory control was developed by Rudie and Wonham~\cite{RW92}. It is based on the idea to distribute the actuator and sensor capabilities among several local supervisors. Each supervisor issues a control decision based on its own observation of the system.
  The global control action is then given by a fusion rule on the local control actions.
  
  To give an example, consider a traffic system of a region with many elements (crossroads, tunnels), traffic lights and information boards. One of the goals may be to make the traffic fluent and prevent traffic jams in case of unexpected circumstances. The elements are observed (automatically or by humans) to provide information about the situation. One element may be observed by several supervisors corresponding to, e.g., different directions. The overall system is modeled as a single system including all dependences between the elements. The control decision is made based on all observations. Thus, although there are local supervisors that do not observe any problem, they need to react based on the observations of other supervisors, for instance, to close all the streets accessing a blocked tunnel.
  
  There is an important motivation for decentralized supervisory control of DES that do not a priori have a modular structure. It is well known that the abstraction of timed automata into region (zone) automata does not preserve the modular structure. Similarly, it is to be expected that the discretization of a hybrid system does not preserve its modular structure. Then the original structure of a hybrid system is lost and we have to face the decentralized supervisory control problem instead of the modular supervisory control problem.
  
  There are many different control policies based on two elementary ones: {\em conjunctive and permissive\/} (C\,\&\,P) and {\em disjunctive and antipermissive\/} (D\,\&\,A). For any decentralized control architecture, a corresponding notion of {\em coobservability\/} was proposed, which together with {\em controllability} form the necessary and sufficient conditions to achieve a specification by the controlled system.
  Nowadays, there are advanced architectures, such as an architecture with conditional decision (inferencing)~\cite{RickerR07,YL04} or multi-level inferencing~\cite{KT07,TKU08}. A general approach consisting in several decentralized supervisory control architectures running in parallel was proposed by Chakib and Khoumsi~\cite{CK11}. 

  Another approach to ensure coobservability is to extend locally observable events via communication among local supervisors. There exist decentralized control problems that cannot be solved without communication. Decentralized control with communicating supervisors, where an occurrence of transitions visible to one supervisor can be communicated to other supervisors, has been studied~\cite{BL00,RR00,RudieLL03}. 
    
  So far, almost all results available in the literature are only existential. There exist a few papers providing constructive results to compute a controllable and coobservable sublanguage of a specification, but the problem is in general computationally difficult. The existence of local supervisors enforcing the safety specification is decidable if nonblockingness is not required (e.g., for prefix-closed languages). However, if the marked language of the controlled system has to be included in the specification so that the controlled system is nonblocking, the existence of such local supervisors is undecidable~\cite{Th05,Tr04}.

  We focus on the computation of a controllable and coobservable sublanguage by using additional communications among the supervisors. Our study is restricted to the original (C\,\&\,P) architecture \cite{RW92,YLL02} and motivated by the relationship between decentralized and modular supervisory control and their key concepts: C\,\&\,P coobservability and separability. This relationship is investigated in Komenda et al.~\cite{KMP06}, where the decentralized framework is plugged into the modular framework. The approach is based on the concurrent (separable) over-approximation of the plant. In decentralized control, there is no assumption on the structure of the plant. Therefore, both the system and the specification are replaced with their infimal separable superlanguages. However, in the likely case the specification fails to be separable, one only computes a solution for this new specification, which often fails to be included in the specification. Otherwise stated, it is assumed that the constructed sublanguage is included in the specification~\cite{KMP06}.

  This is the point, where our approach comes into the picture. We overcome the problem of an inseparable specification by making it separable via communication (using the notion of {\em conditional decomposability}~\cite{KMvS12cd}). The approach of Komenda et al.~\cite{KMP06} further requires {\em mutual controllability\/} to ensure coobservability of the separable over-approximation of the specification. This can be omitted in our approach, although we show that it is useful to ensure the centralized optimality.
  
  Decomposable over-approximations were also considered by Jiang and Kumar~\cite{JKumar00}, where decomposability of the specification is an additional assumption, whereas in our approach, it is enforced by the construction via additional communications.
  
  In this paper, we construct a controllable and coobservable sublanguage of a specification with respect to possibly extended local observations (enriched by communication). The idea is different from the computation of purely coobservable sublanguages in the literature. First, we extend the set of locally observable events via communication of events observable by other supervisors. To compute such an extension, we use the technique for conditional decomposability~\cite{KMvS12cd}. Then we use the extended alphabets to compute the local supervisors in a fully decentralized way. 
  We show that our construction guarantees coobservability (Theorem~\ref{main}) as soon as the local supervisors are nonconflicting (in particular if they are prefix-closed). We then state two sufficient conditions under which the solution coincides with the centralized optimal solution. One approach uses mutual controllability (Theorem~\ref{optfoc}), which is a condition considering the structure of the plant, and the other the observer and LCC conditions (Theorem~\ref{thm_inc_opt_sol}), which rather considers the structure of the supervisors.
  Since our approach is not restricted to prefix-closed languages, we further show how to handle the case of conflicting supervisors (Theorem~\ref{thm22}). The complexity of our approach is briefly discussed in Section~\ref{complexity}. We further show that separability is PSPACE-complete (Theorem~\ref{thm6}), which generalizes a result on the complexity of decomposability. 
  
  Although our approach is different from the existing approaches, the rough idea is analogous to that used for decentralized synthesis without communication~\cite{KozakW95,Takai98}. The approach of Koz{\' a}k and Wonham~\cite{KozakW95} is similar to ours in that it computes fully decentralized solutions that guarantee coobservability. It essentially projects the centralized supervisor, whereas we compute local supervisors and coobservability is guaranteed by construction (from separability that is granted by distributed computation). The main difference is that we do not project the centralized supervisor, but rather the plant and the specification. 
  Also, the condition of Takai~\cite{Takai98} under which the fully decentralized supervisors achieve the centralized optimal solution (the centralized optimal supervisor must be observable with respect to all locally observable alphabets) is different from our condition that relies on structural properties of the plant: such as observer and local control consistency or mutual controllability and observability of projections of the plant.
  
  Our approach is also related to the topic on sensor selection~\cite{JiangKG03,RohloffKK06,WangGLL11,WangLL08}, where there is a maximal observable event set and the sensors can be turned on and off. In our approach, we have static observations, rather than dynamic, and communicate them among supervisors, where needed. Our work can be extended to dynamic observations in the future.

  This work is an extended version of the conference paper~\cite{allerton2013}. We extend our approach to non-prefixed-closed languages and to partial observations, compare it to the centralized optimal solution, and include a discussion on the complexity of related problems.

\section{Preliminaries}\label{sec:cc}
  We assume that the reader is familiar with decentralized supervisory control of discrete-event systems~\cite{CL08,Won12}.
  Let $A$ be a finite nonempty set (an {\em alphabet}), and let $A^*$ denote the set of all finite words over $A$; the empty word is denoted by $\eps$. A {\em language\/} over $A$ is a subset of $A^*$. The {\em prefix closure\/} of a language $L$ over $A$ is the set $\overline{L}=\{w\in A^* \mid \text{there exists } u \in A^* \text{ such that } wu\in L\}$. A language $L$ is {\em prefix-closed\/} if $L=\overline{L}$.

  A {\em generator\/} is a quintuple $G=(Q,A,f,q_0,Q_m)$, where $Q$ is a finite set of {\em states}, $A$ is an alphabet (of {\em events}), $f\colon Q \times A \to Q$ is a {\em partial transition function}, $q_0 \in Q$ is the {\em initial state}, and $Q_m\subseteq Q$ is a set of {\em marked states}. The transition function $f$ can be extended to the domain $Q \times A^*$ in the usual way. The language {\em generated\/} by $G$ is the set $L(G) = \{s\in A^* \mid f(q_0,s)\in Q\}$ and the language {\em marked\/} by $G$ is the set $L_m(G) =\{s\in A^* \mid f(q_0,s)\in Q_m\}$.
  The paper is restricted to {\em regular languages}, that is, languages marked by a generator.

  A {\em projection\/} $P_o\colon A^* \to A_o^*$ is a morphism defined by $P_o(a) = \eps$, if $a\in A\setminus A_o$, and $P_o(a)= a$, if $a\in A_o$. It is extended (as a morphism for concatenation) from events to words by induction. The {\em inverse image} of $P_o$ is defined as $P_o^{-1}(a)=\{s\in A^* \mid P_o(s) = a\}$. These definitions can naturally be extended to languages. For two alphabets $A_x,A_y\subseteq A$, we use the notation $P_{y}^{x}$ to denote a projection from $A_x^*$ to $A_y^*$ and we write simply $P_y$ if $A_x = A$.
  
  A {\em synchronous product\/} of languages $L_i\subseteq A_i^*$, for $i=1,\ldots,n$, is defined as $\bigparallel_{i=1}^n L_i=\bigcap_{i=1}^n P_i^{-1}(L_i) \subseteq A^*=\left(\bigcup_{i=1}^n A_i\right)^*$, where $P_i\colon A^*\to A_i^*$ is a projection. Languages $L_i$ are {\em synchronously nonconflicting\/} if $\overline{\|_{i=1}^{n} L_i} = \|_{i=1}^{n} \overline{L_i}$. 

  Let $L$ be a prefix-closed language over an alphabet $A$. A language $K\subseteq L$ is {\em observable\/} with respect to $L$ and projection $P_{o}\colon A^* \to A_{o}^*$ if, for all $s\in \overline{K}$ and $a \in A_c$, $sa \notin \overline{K}$ and $sa \in L$ implies that $P_{o}^{-1}P_{o}(s)\{a\} \cap \overline{K} = \emptyset$. Language $K$ is {\em normal\/} with respect to $L$ and projection $P_{o}$ if $\overline{K} = P_{o}^{-1} P_{o}(\overline{K}) \cap L$.

\subsection{Decentralized Supervisory Control}
  A {\em controlled generator\/} over an alphabet $A$ is a structure $(G,(A_{c,i})_{i=1}^{n},(A_{o,i})_{i=1}^{n})$, where $G$ is a generator over $A$, $A_{c,i}\subseteq A$ are sets of {\em locally controllable events}, and $A_{o,i} \subseteq A$ are sets of {\em locally observable events}. Let $A_c = \bigcup_{i=1}^n A_{c,i}$ denote the set of controllable events, $A_o = \bigcup_{i=1}^n A_{o,i}$ the set of observable events, $A_{uc}=A\setminus A_{c}$ the set of uncontrollable events, and $A_{uo}=A\setminus A_{o}$ the set of unobservable events. Projections to locally observable events $A_{o,i}$ are denoted by $P_{o,i}\colon A^* \to A_{o,i}^*$. 

  Let $\Gamma_i=\{\gamma\subseteq A \mid \gamma\supseteq (A\setminus A_{c,i})\}$ be a set of local control patterns. A supervisor $S_i$ is a mapping $S_i\colon P_{o,i}(L(G))\to \Gamma_i$, where $S_i(s)$ is the set of locally enabled events if $S_i$ observes $s\in A_{o,i}^*$. The global control law $S$ is the conjunction of local supervisors $S_i$ given by $S(w) =\bigcap_{i=1}^n S_i(P_{o,i}(w))$ for $w\in A^*$.
  The {\em closed-loop system\/} is the smallest language $L(S/G)$ such that $\eps \in L(S/G)$ and if $s \in L(S/G)$, $sa\in L(G)$ and $a \in S(s)$, then $sa \in L(S/G)$. 
  Control objectives of decentralized control are defined using a specification language $K$. Let $L_m(S/G) = L(S/G)\cap K$. If $\overline{L_m(S/G)}=L(S/G)$, the closed-loop system is called {\em nonblocking}. The goal of decentralized control is to find supervisors $(S_i)_{i=1}^{n}$ such that $L_m(S/G)=K$ and $\overline{L_m(S/G)}=\overline{K}$.

  Necessary and sufficient conditions to achieve a specification by a joint action of local supervisors are controllability and coobservability~\cite{RW92}.
  Let $L$ be a prefix-closed language over $A$.
  A language $K\subseteq L$ is {\em controllable\/} with respect to $L$ and the set of uncontrollable events $A_{uc}$ if $\overline{K}A_{uc}\cap L\subseteq \overline{K}$. 
  A language $K\subseteq L$ is {\em coobservable\/} with respect to $L$ and the sets of locally observable events $(A_{o,i})_{i=1}^{n}$ if for all $s\in \overline{K}$, $a\in A_c$, and $sa \in L\setminus \overline{K}$, there exists $i\in \{1,2,\ldots,n\}$ such that $a\in A_{c,i}$ and $(P_{o,i}^{-1}(P_{o,i}(s))\{a\}\cap \overline{K}=\emptyset$.
  Intuitively, if, after a word $s$ from the specification, the extension by an event $a$ is illegal (it does not exist in the specification but exists in the plant), then there must exist at least one local supervisor $S_i$ that can issue the decision ``disable the event $a$''. 

  The control law of local supervisors associated to the C\,\&\,P architecture is called permissive, since the default action is to enable an event whenever a supervisor has an ambiguity what to do with it. Specifically, the control law of supervisor $S_i$ on $s$ is defined as $S_i(s)=(A\setminus A_{c,i}) \cup \{ a\in A_{c,i} \mid \text{there exists } s'\in K \text{ with } P_{o,i}(s')=P_{o,i}(s) \text{ and } s'a\in K\}$.
  With the permissive local policy, we always achieve all words in the specification. The concern is then {\em safety}, expressed by coobservability.
  
  Let $X\subseteq A$ be an alphabet. In the rest of the paper, we use the convention to define the set of uncontrollable events of $X$ as $X_{uc} = X \cap A_{uc}$. Similarly, we define the set of controllable events of $X$ as $X_c=X\cap A_c$, the set of observable events of $X$ as $X_{o}=X\cap A_o$, and the set of unobservable events of $X$ as $X_{uo}=X\cap A_{uo}$.

\section{Main Idea of our Approach}
  We now present the main idea of our approach to compute a controllable and coobservable sublanguage of a specification language using communication and results of modular supervisory control.
  
  Let $(G,(A_{c,i})_{i=1}^{n},(A_{o,i})_{i=1}^{n})$ be a controlled generator over an alphabet $A$. For simplicity, we denote $L=L(G)$. Let $K\subseteq L$ be a specification language over $A$. If the local supervisors do not observe all events of $A$, that is, $A_{uo}$ is nonempty, we consider an arbitrary decomposition of $A_{uo}$ into (not necessarily disjoint) local sets $A_{uo,i}$, such that the union of all $A_{uo,i}$ results in $A_{uo}$. The alphabet $B_i$ of supervisor $S_i$ contains all events from $A_{o,i} \cup A_{uo,i}$ and may be further extended with other events via communication. The union of all the alphabets $B_i$ results in $A$.
  In Section~\ref{sec:separability}, we suggest a procedure how to obtain such a decomposition of $A_{uo}$. Namely, for every alphabet $A_{o,i}$, we make use of the procedure {\sc Rcd} defined on page~\pageref{RCDpage} below that computes an extension alphabet $\Sigma_i \subseteq A$. This extension can be decomposed into observable and unobservable events $\Sigma_{o,i}$ and $\Sigma_{uo,i}$ with respect to global observable and unobservable alphabets $A_o$ and $A_{uo}$, respectively. The alphabet $B_i$ of supervisor $S_i$ is then the union of alphabets $A_{o,i}$, $\Sigma_{o,i}$, and $\Sigma_{uo,i}$. The events of the alphabet $A_{o,i} \cup \Sigma_{o,i}$ are the events observed by supervisor $S_i$ extended with communications.
  
  The idea of our approach is to compute local languages (supervisors) $\R_i$ over the alphabets $B_i$ such that their synchronous product $\R = \|_{i=1}^{n} \R_i$ is a sublanguage of $K$ controllable and coobservable with respect to $L$. Although there are well-known conditions on local languages in modular supervisory control that ensure that their synchronous product is controllable, cf. Lemma~\ref{feng} in the appendix, conditions on local languages that ensure coobservability of their synchronous product are not known. We now identify two such sufficient conditions in Theorem~\ref{newthm}.
  
  \begin{theorem}\label{newthm}
    Let $L$ be a prefixed-closed language over $B = \bigcup_{i=1}^{n} B_i$ and assume that $B_{o,i}\cap B_c\subseteq B_{c,i}$, for $i=1,\ldots,n$. Let $M \subseteq L$ be a language such that $\overline{M} = \|_{i=1}^{n} \overline{M_i}$, where $M_i$ is a language over $B_i$. If
    \begin{enumerate}
      \item either $M_i$ is normal with respect to $P_i(L)$ and $P^i_{o,i}$, for all $i=1,\ldots,n$,
      \item or $B_c \subseteq B_o$ and $M_i $ is observable with respect to $P_i(L)$ and $P^i_{o,i}$, for all $i=1,\ldots,n$,
    \end{enumerate}
    then $M$ is coobservable with respect to $L$ and $(B_{o,i})_{i=1}^{n}$.
  \end{theorem}
  \begin{proof}
    For the sake of contradiction, assume that language $\overline{M}$ is not coobservable with respect to $L$ and $(B_{o,i})_{i=1}^{n}$. Then there exist $s\in\overline{M}$ and $a\in B_c$ such that $sa \in L \setminus \overline{M}$ and, for each $i \in \{1,2,\ldots,n\}$, either $a \notin B_{c,i}$ or $P_{o,i}^{-1}P_{o,i}(s) \{a\} \cap \overline{M} \neq \emptyset$. Let $t = P_i(s) \in P_i(\overline{M}) \subseteq \overline{M_i}$. Then $sa\in L$ implies that $tP_i(a)\in P_i(L)$. We show below that $sa \in P_i^{-1}(\overline{M_i})$. It then completes the proof, since $sa \in \bigcap_{i=1}^n P_i^{-1}(\overline{M_i}) = \overline{M}$, which is a contradiction with $sa\in L\setminus\overline{M}$.
    
    If $a\in B_{c,i}$, then there exists $s_i$ such that $s_ia \in \overline{M}$ and $P_{o,i}(s_i)=P_{o,i}(s)$. Let $t_i=P_i(s_i)\in P_i(\overline{M}) \subseteq \overline{M_i}$. Then $s_ia \in \overline{M}$ implies that $t_ia\in P_i(\overline{M}) \subseteq \overline{M_i}$. Moreover, $P^i_{o,i}(t) = P^i_{o,i}(P_i(s)) = P_{o,i}(s) = P_{o,i}(s_i) = P^i_{o,i}(P_i(s_i)) = P^i_{o,i}(t_i)$. Since $t_i a\in \overline{M_i}$ and $t_i a\in (P_{o,i}^i)^{-1}P_{o,i}^i(t)\{a\}$, observability of $M_i$ with respect to $P_i(L)$ and $P_{o,i}^i$ implies that $ta = P_i(sa) \in \overline{M_i}$, that is, $sa\in P_i^{-1}(\overline{M_i})$.
    
    If $a\notin B_{c,i}$, then $a\not\in B_{o,i}\cap B_c$. Since $a\in B_c$, we have that $a\not\in B_{o,i}$. If $a\not\in B_{uo,i}$, then $a\not\in B_i$. Hence $P_i(a)=\eps$, and $P_i(sa)=P_i(s)\in P_i(\overline{M})\subseteq \overline{M_i}$ implies that $sa\in P_i^{-1}(\overline{M_i})$.
    If $a\in B_{uo,i}$, then normality of $M_i$ with respect to $P_i(L)$ implies that $\overline{M_i} = (P_{o,i}^{i})^{-1} P_{o,i}^{i} (\overline{M_i}) \cap P_i(L)$. Since $t\in \overline{M_i}$ and $P_{o,i}^{i}(a) = \eps$, we have that $ta \in (P_{o,i}^{i})^{-1} P_{o,i}^{i} (\overline{M_i})$. Then $ta \in P_i(L)$ implies that $ta \in \overline{M_i}$, and $ta = P_i(sa)$ then gives that $sa \in P_i^{-1}(\overline{M_i})$.
    This completes the proof for $M_i$ normal.

    If $M_i$ is not normal, we have that $B_c \subseteq B_o$ and $M_i$ is observable. Then $a \in B_c$ implies that $a \not\in B_{uo}$, hence the only possible case is $a \notin B_{uo,i}$, which we have already shown above.
  \end{proof}

  The way we compute the languages $\R_i$ is as follows. We decompose specification $K$ in such a way that $K = \|_{i=1}^{n} K_i$, where $K_i$ are languages over $B_{i}$, and over-approximate $L$ by the synchronous product of its projections $P_{i}(L)$ on alphabets $B_i$. The condition required on $K$ does not always hold and is equivalent to the notion of {\em separability\/} defined below. How to construct the alphabets $B_i$ so that $K$ satisfies the separability condition is discussed in Section~\ref{sec:separability}. The languages $\R_i$ are then computed locally as sublanguages or superlanguages of $K_i$ that satisfy the sufficient conditions (Lemma~\ref{feng} and Theorem~\ref{newthm}) that make their synchronous product $\R$ controllable, coobservable and included in $K$. These computations are discussed in Section~\ref{sec:comp}.

  In Theorem~\ref{newthm}, we assume that if a supervisor observes a controllable event, it can also control it; that is, $B_{o,i}\cap B_c\subseteq B_{c,i}$. This is a new condition that deserves a discussion. Rudie and Wonham~\cite{RW92} showed that under the assumption that a supervisor can always observe the events it can control, that is, $B_{c,i}\subseteq B_{o,i}$, decomposability (a generalization of separability) is equivalent to coobservability. Our condition $B_{o,i}\cap B_c\subseteq B_{c,i}$ is weaker in the sense that it does not require that $B_{c,i}$ is included in $B_{o,i}$. Similarly, the assumption $B_c \subseteq B_o$ in case~2 of Theorem~\ref{newthm} does not mean that $B_{c,i}$ is included in $B_{o,i}$. It only requires that every controllable event is observed by one of the supervisors. 

  To justify our assumption, let $a$ be a controllable event that is observable by a supervisor $S_i$. If $a$ is not physically controllable by $S_i$, we can still make use of the advantage that $S_i$ observes $a$. Namely, $S_i$ may provide information about $a$ as if $a$ was controllable for it. The fusion rule (global supervisor $S$) then decides which events need to be disabled in the current situation, and communicates this decision back to the supervisors. If $S_i$ requires that $a$ needs to be disabled, the global supervisor will require to disable $a$. Since $a$ is controllable, there must be a local supervisor $S_j$ that can physically control $a$. Then supervisor $S_j$ will take care of disabling $a$. Another view is that if $S_i$ finds out that $a$ needs to be disabled, it communicates this observation directly to $S_j$, which takes the corresponding actions.

\section{Separability and Communication}\label{sec:separability}
  In this section, we define the notion of separability and suggest a procedure to construct the alphabets $B_i$ so that $K$ is separable with respect to $(B_i)_{i=1}^{n}$ as required in our approach.
  
  A conceptually simpler condition than coobservability is known in the literature as decomposability~\cite{RudieW90}. A language $K$ over $A$ is {\em decomposable\/} with respect to alphabets $(A_{i})_{i=1}^{n}$ and $L$ if $K = ( \bigcap_{i=1}^{n} P_i^{-1}P_{i}(K) ) \cap L$, where $P_i$ is the projection from $A^*$ to $A_i^*$. The inclusion $K \subseteq \bigcap_{i=1}^{n} P_i^{-1}P_{i}(K) \cap L$ holds true whenever $K\subseteq L$. A special case of decomposability for $L=A^*$ is known as {\em separability\/}~\cite{WH1991}. For $A = \bigcup_{i=1}^{n} A_i$, we can replace intersection in the definition by parallel composition. Namely, language $K$ is separable with respect to alphabets $(A_i)_{i=1}^{n}$ if $K =  \|_{i=1}^{n} P_{i}(K)$. As already pointed out, separability of $K$ with respect to alphabets $(A_i)_{i=1}^{n}$ is equivalent to the existence of languages $K_i$ over $A_i$ such that the synchronous product of $K_i$ results in $K$. Then $P_i(K)$ is included in $K_i$, hence the languages $P_i(K)$ are the minimal (infimal) languages (with respect to inclusion) whose synchronous product results in $K$.

  From the computational point of view, the first question is the complexity of deciding whether language $K$ is separable with respect to alphabets $(A_i)_{i=1}^{n}$. We show that the problem is PSPACE-complete. A decision problem is PSPACE-complete if it can be solved in polynomial space with respect to the size of the input and if every problem that can be solved in polynomial space can be reduced to it in polynomial time. A proof of the following result can be found in the appendix. 

  \newcounter{thm6again}
  \setcounter{thm6again}{\value{theorem}}
  \begin{theorem}\label{thm6}
    The following problem is PSPACE-complete.
    \begin{itemize}
      \item[] \textsc{Input:} Alphabets $A_1,A_2,\ldots,A_n$ and a generator $H$ over $\bigcup_{i=1}^{n} A_i$.
      \item[] \textsc{Output:} {\tt Yes} if and only if $L_m(H)$ is separable with respect to $(A_i)_{i=1}^{n}$.
    \end{itemize}
  \end{theorem}
  
  The size of the input is the number of states and transitions of the generator $H$ and the size of the $n$ alphabets $A_1,\ldots, A_n$. The verification can be done in polynomial time by a direct computation if the number of alphabets is restricted by a constant. Therefore, the unrestricted number of alphabets (supervisors) is what makes the problem computationally difficult.
  
  Separability is a special case of decomposability where $L$ is universal. As a consequence, decomposability of $K$ with respect to $(A_{i})_{i=1}^{n}$ and $L$ is PSPACE-complete even if $L=A^*$. This generalizes a result that can be derived from the literature. Namely, coobservability of $K$ with respect to $(A_{i})_{i=1}^{n}$ and $L$ is known to be PSPACE-complete~\cite{RohloffYL03}. Under some assumptions~\cite[Proposition~4.3]{RW92}, decomposability is equivalent to coobservability. Since the reduction by Rohloff et al.~\cite{RohloffYL03} satisfies these assumptions, decomposability of $K$ with respect to $(A_{i})_{i=1}^{n}$ and $L$ is PSPACE-complete. In those proofs, however, $L$ is different from $A^*$, hence our theorem generalizes this result.

  Another question, which we have to face, is what to do if language $K$ is not separable with respect to $(A_i)_{i=1}^{n}$. In this case, it would be natural to take a maximal (with respect to inclusion) sublanguage of $K$ that is separable with respect to $(A_i)_{i=1}^{n}$. Unfortunately, Lin et al.~\cite{LinSSWS14} have shown that to find such a maximal sublanguage is not algorithmically possible.

  To overcome these issues -- high complexity and undecidability -- we use the notion of conditional decomposability~\cite{KMvS12cd}. A language $K$ is {\em conditionally decomposable\/} with respect to alphabets $(A_{i})_{i=1}^{n}$ and an alphabet $\Sigma$ if $\Sigma$ contains all shared events, that is, $\bigcup_{i\neq j} (A_i\cap A_j)\subseteq \Sigma \subseteq \bigcup_{i=1}^{n} A_i$, and $K$ is separable with respect to alphabets $(A_i\cup\Sigma)_{i=1}^{n}$, that is, $K = \|_{i=1}^{n} P_{i+\Sigma}(K)$, where $P_{i+\Sigma}$ denotes the projection from $A^*$ to $(A_{i}\cup \Sigma)^*$. Conditional decomposability thus requires to find an alphabet $\Sigma$ containing all shared events such that $K$ is separable with respect to $(A_i\cup \Sigma)_{i=1}^{n}$. 
  
  Compared to separability, there are two advantages of conditional decomposability. 
  First, every language can be made conditionally decomposable by finding a convenient alphabet $\Sigma$. Such an alphabet always exists; indeed, one could take $\Sigma = \bigcup_{i=1}^{n} A_i$, but the aim is to find a reasonably small alphabet. This advantage of conditional decomposability helps us overcome the undecidable issue of finding a maximal nonempty separable sublanguage.
  The second advantage is a lower complexity of checking conditional decomposability. To check whether $K$ is conditionally decomposable with respect to $(A_i)_{i=1}^{n}$ and $\Sigma$ can be done in polynomial time, compared to PSPACE for separability. What allows this efficiency is the assumption that $\Sigma$ contains all shared events. The following theorem is a generalization of a result obtained for pairwise disjoint alphabets by Willner and Heymann~\cite{WH1991}.
  \begin{theorem}[\cite{KMvS12cd}]\label{thmgen}
    Let $K$ be a language represented as a generator, and let $(A_i)_{i=1}^{n}$ and $\Sigma$ be alphabets. The problem to decide whether $K$ is conditionally decomposable with respect to $(A_i)_{i=1}^{n}$ and $\Sigma$ can be solved in polynomial time.
  \end{theorem}

  Recall that our idea to compute the controllable and coobservable sublanguage involves an over-approximation of the plant language $L$ by a new modular plant $\|_{i=1}^{n}P_{i+\Sigma}(L)$. We show in the following lemma that it is better to consider the projections $P_{i+\Sigma}(L)$ rather than $P_{i}(L)$ used in Komenda et al.~\cite{KMP06} because the larger the extension $\Sigma$, the better the over-approximation of $L$.
  
  \begin{lemma}\label{lemInclusion}
    Let $(A_i)_{i=1}^{n}$ be alphabets, and let $L$ be a language over the alphabet $A = \bigcup_{i=1}^{n} A_i$. Let $\Sigma \subseteq A$ be an alphabet, and let $P_i\colon A^* \to A_i^*$ and $P_{i+\Sigma}\colon A^* \to (A_i\cup\Sigma)^*$ be projections. Then $L \subseteq \|_{i=1}^{n} P_{i+\Sigma}(L) \subseteq \|_{i=1}^{n} P_i(L)$.
  \end{lemma}
  \begin{proof}
    The first inclusion holds for any projection. To prove the other inclusion, notice that it holds that $P_{i+\Sigma}(L) \subseteq (P_{i}^{i+\Sigma})^{-1} (P_{i}^{i+\Sigma} (P_{i+\Sigma}(L))) = (P_{i}^{i+\Sigma})^{-1}(P_i(L))$, where $P_i^{i+\Sigma}$ is the projection $P_i$ restricted to the domain $(A_i\cup\Sigma)^*$. Then $P_{i+\Sigma}^{-1}(P_{i+\Sigma}(L)) \subseteq P_{i+\Sigma}^{-1}(P_{i}^{i+\Sigma})^{-1}(P_i(L)) = P_{i}^{-1} P_i(L)$, which completes the proof.
  \end{proof}

  It may seem that the largest $\Sigma$ is the best choice. However, a larger $\Sigma$ means more communication or more sensors to observe the system (the local supervisors need to observe more). On the other hand, to compute a minimal extension $\Sigma$ with respect to the cardinality is an NP-hard problem~\cite{KMvS13JDEDS}. Nevertheless, there is an algorithm to find an acceptable extension in polynomial time~\cite{KMvS12cd}. In general, to find a suitable extension in a reasonable time is an interesting research topic. The choice of $\Sigma$ can be influenced by several factors, such as the price of sensors, the (im)possibility to observe an event (by a specific local supervisor) etc.
  
  Even if we consider the minimal extension $\Sigma$, it may happen that too many events are forced to be communicated between all supervisors though it is not needed. We demonstrate this in the following example. It shows that it is more convenient to search for some local alphabets so that the specification is separable with respect to them.

  \begin{example}\label{ex20}
    Consider the language $K$ over $A=\{a,b,c,d,e,f\}$, whose generator is depicted in Fig.~\ref{spec}, and the alphabets $A_{o,1}=\{a,e\}$, $A_{o,2}=\{b,e\}$, $A_{o,3}=\{c,f\}$, and $A_{o,4}=\{d,f\}$.
    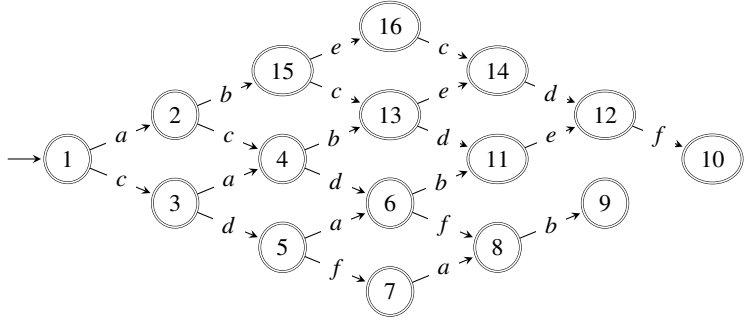
\begin{figure}
      \centering
      \begin{tikzpicture}[baseline,->,>=stealth,shorten >=1pt,node distance=2cm,
        state/.style={ellipse,minimum size=6mm,very thin,draw=black,initial text=},
        every node/.style={fill=white,accepting,font=\small},
        bigloop/.style ={shift={(0,0.0)},text width=.6cm,align=center},
        bigloopd/.style={shift={(0,0.0)},text width=.6cm,align=center}]
        \node[state,initial]  (1)               {$1$};
        \node[state]          (2) [below right of=1,above of=1]  {$2$};
        \node[state]          (3) [above right of=1,below of=1]  {$3$};
        \node[state]          (4) [below right of=3,above of=3]  {$4$};
        \node[state]          (5) [above right of=3,below of=3]  {$5$};
        \node[state]          (15) [below right of=2,above of=2]  {$15$};
        \node[state]          (16) [below right of=15,above of=15]  {$16$};
        \node[state]          (13) [above right of=15,below of=15]  {$13$};
        \node[state]          (6) [below right of=5,above of=5]  {$6$};
        \node[state]          (7) [above right of=5,below of=5]  {$7$};
        \node[state]          (14) [below right of=13,above of=13]  {$14$};
        \node[state]          (11) [above right of=13,below of=13]  {$11$};
        \node[state]          (8) [above right of=6,below of=6]  {$8$};
        \node[state]          (12) [below right of=11,above of=11]  {$12$};
        \node[state]          (9) [above right of=11,below of=11]  {$9$};
        \node[state]          (10) [above right of=12,below of=12]  {$10$};

        \foreach \from/\to/\pis in {1/2/a, 1/3/c, 2/15/b, 2/4/c, 3/4/a, 3/5/d, 4/13/b, 4/6/d, 5/6/a, 5/7/f, 6/11/b, 6/8/f, 7/8/a, 8/9/b, 11/12/e, 12/10/f, 13/14/e, 13/11/d, 14/12/d, 15/16/e, 15/13/c, 16/14/c}
          \path (\from) edge node{$\pis$} (\to);
      \end{tikzpicture}
      \caption{Generator for language $K$ of Example~\ref{ex20}}
      \label{spec}
    \end{figure}
    To make $K$ conditionally decomposable with respect to $(A_{o,i})_{i=1}^{4}$ and $\Sigma$, the extension $\Sigma$ must contain at least all shared events, that is, $e$ and $f$; actually, $\Sigma=\{e,f,a,c\}$ is a minimal extension making $K$ conditionally decomposable with respect to $(A_{o,i})_{i=1}^{4}$ and $\Sigma$.
    However, the reader may notice that $a$ needs to be communicated only between the supervisors $S_1$ and $S_2$, whereas $c$ needs to be communicated only between the supervisors $S_3$ and $S_4$. Specifically, $K$ is conditionally decomposable with respect to alphabets $A_{o,1}\cup A_{o,2}$, $A_{o,3}\cup A_{o,4}$ and $\Sigma_{all}=\{f\}$; that is, $K = P_{\{a,b,e,f\}}(K) \parallel P_{\{c,d,f\}}(K)$. Having this, notice that language $P_{\{a,b,e,f\}}(K)=\overline{\{abef,afb,fab\}}$ is conditionally decomposable with respect to $A_{o,1}\cup\Sigma_{all}$, $A_{o,2}\cup\Sigma_{all}$ and $\Sigma_{1,2}=\{a,e,f\}$, and language  $P_{\{c,d,f\}}(K)=\overline{\{cdf\}}$ is conditionally decomposable with respect to $A_{o,3}\cup\Sigma_{all}$, $A_{o,4}\cup\Sigma_{all}$ and $\Sigma_{3,4}=\{c,f\}$. This means that $a$ and $c$ are communicated only locally and only $f$ is communicated globally. In other words, $K$ is separable with respect to alphabets $\{a,e,f\}$, $\{a,b,e,f\}$, $\{c,f\}$, and $\{c,d,f\}$. 
    
    The reader can compare this with the purely conditional decomposable case computing the single extension $\Sigma=\{e,f,a,c\}$ making the language $K$ separable with respect to $(A_{o,i}\cup\Sigma)_{i=1}^{n}$, that is, with respect to alphabets $\{a,c,e,f\}$, $\{a,b,c,e,f\}$, $\{a,c,e,f\}$, and $\{a,c,d,e,f\}$.
  \qed\end{example}
  
  In Example~\ref{ex20} we need to check separability of $K$ with respect to alphabets $A_{o,1}\cup A_{o,2}\cup\Sigma_{all}=\{a,b,e,f\}$ and $A_{o,3}\cup A_{o,4}\cup\Sigma_{all}=\{c,d,f\}$. According to Theorem~\ref{thmgen}, this can be done in polynomial time. Moreover, the alphabets $\Sigma_{all}$, $\Sigma_{1,2}$ and $\Sigma_{3,4}$ are computed in polynomial time. This suggests the following refinement of the conditional decomposability procedure. 

  \paragraph{Procedure: \sc Refined Conditional Decomposability (Rcd)\\}\label{RCDpage}
  Let $(A_i)_{i=1}^{n}$ be alphabets such that $A_i\neq A_j$ for $i\neq j$, and let $K$ be a language over $\bigcup_{i=1}^{n} A_i$. We define the equivalence relation $\sim$ as the minimal equivalence relation such that $A_i \sim A_j$ if $A_i$ and $A_j$ are not disjoint ($A_i\cap A_j\neq\emptyset$). The following steps of the procedure are illustrated in Example~\ref{ex21} below.
  \begin{enumerate}
    \item Let $A_i/_{\sim} = \{A_j \mid A_j \sim A_i, 1\le j\le n\}$ be the equivalence class of alphabets equivalent to $A_i$. Let $C_i = \bigcup_{A_j \in A_i/_{\sim}} A_j$ be the set of all events that appear in $A_i/_{\sim}$. Let $\{C_{k_1},\ldots,C_{k_m}\}=\{C_1,\ldots,C_n\}$, that is, we remove duplicates, hence we have $C_{k_i} \cap C_{k_j} = \emptyset$ for $k_i\neq k_j$.
    
    \item We first compute a global extension $\Sigma_{all}$ of events shared by the sets $(C_{k_i})_{i=1}^{m}$, which makes $K$ conditionally decomposable with respect to $(C_{k_i})_{i=1}^{m}$ and $\Sigma_{all}$. This gives
    \[
      K = \|_{i=1}^{m} P_{C_{k_i}\cup\Sigma_{all}}(K)\,.
    \]
    
    \item Then, for each local part $A_{k_i}/_{\sim}$, $i=1,2,\ldots,m$, we make the language $P_{C_{k_i}\cup \Sigma_{all}}(K)$ conditionally decomposable with respect to $D_{k_i} = \{A_j\cup\Sigma_{all} \mid A_j\in A_{k_i}/_{\sim}\}$ and $\Sigma_{k_i}$. The set notation is used here to eliminate duplicates. If $D_{k_i}$ is a singleton, we set $\Sigma_{k_i}=\Sigma_{all}$. This then gives that
    \[
      P_{C_{k_i}\cup\Sigma_{all}}(K) 
      = \|_{C\in D_{k_i}} P_{C\cup\Sigma_{k_i}}(K)\,.
    \]
    Notice that if $C,C'\in D_{k_i}$, then $\Sigma_{all}\subseteq C\cap C'$, which implies that $\Sigma_{all}\subseteq \Sigma_{k_i}$.
    
    \item Finally, for every $A_j \in A_{k_i}/_{\sim}$, we set $\Sigma_j = \Sigma_{k_i}$.
  \end{enumerate}
  
  We now illustrate this procedure on a previous example.
  \begin{example}\label{ex21}
    Consider Example~\ref{ex20}. In this case, we have that $A_{o,1}\sim A_{o,2}$ and $A_{o,3}\sim A_{o,4}$. Therefore, $A_1/_{\sim} = A_2/_{\sim} = \{A_{o,1},A_{o,2}\}$ and $A_3/_{\sim} = A_4/_{\sim} = \{A_{o,3}, A_{o,4}\}$. It further gives that $C_1 = C_2 = A_{o,1}\cup A_{o,2}$ and $C_3=C_4=A_{o,3}\cup A_{o,4}$. Then we compute $\Sigma_{all}=\{f\}$ such that $K$ is conditionally decomposable with respect to $C_1$, $C_3$ and $\Sigma_{all}$. In other words, $K$ is separable with respect to $C_1\cup\Sigma_{all}$ and $C_3\cup\Sigma_{all}$. The projections are $P_{C_1\cup\Sigma_{all}}\colon A^* \to \{a,b,e,f\}^*$ and $P_{C_3\cup\Sigma_{all}}\colon A^* \to \{c,d,f\}^*$. Then we compute $\Sigma_1=\{a,e,f\}$ such that language $P_{C_1\cup\Sigma_{all}}(K) = \overline{\{abef,afb,fab\}}$ is conditionally decomposable with respect to $A_{o,1}\cup\Sigma_{all}$, $A_{o,2}\cup\Sigma_{all}$ and $\Sigma_1$, and $\Sigma_3=\{c,f\}$ such that language $P_{C_3\cup\Sigma_{all}}(K) = \overline{\{cdf\}}$ is conditionally decomposable with respect to $A_{o,3}\cup\Sigma_{all}$, $A_{o,4}\cup\Sigma_{all}$ and $\Sigma_3$. Finally, we set $\Sigma_2=\Sigma_1$ and $\Sigma_4=\Sigma_3$. Language $K$ is then separable with respect to $(A_{o,i}\cup\Sigma_i)_{i=1}^{4}$, that is, with respect to alphabets $B_1=\{a,e,f\}$, $B_2=\{a,b,e,f\}$, $B_3=\{c,f\}$, and $B_4=\{c,d,f\}$.
  \qed\end{example}
  
  We now show that the procedure is correct.
  \begin{theorem}
    Let $(A_i)_{i=1}^{n}$ be alphabets such that $A_i\neq A_j$ for $i\neq j$, and let $K$ be a language over $\bigcup_{i=1}^{n} A_i$. Let $(\Sigma_i)_{i=1}^{n}$ be the extensions computed by procedure {\sc Rcd}. Then $K$ is separable with respect to $(A_i\cup\Sigma_i)_{i=1}^{n}$.
  \end{theorem}
  \begin{proof}
    Let $\{C_{k_1},\ldots,C_{k_m}\}=\{C_1,\ldots,C_n\}$. Then $C_{k_i} \cap C_{k_j} = \emptyset$, for $k_i\neq k_j$, since $C_{k_i}$ and $C_{k_j}$ contain events of different classes. After the computation of $\Sigma_{all}$, $K = \|_{i=1}^{m} P_{C_{k_i}\cup\Sigma_{all}}(K)$. Then, for each $D_{k_i}$, $i=1,\ldots,m$, we compute $\Sigma_{k_i}$ such that $P_{C_{k_i}\cup\Sigma_{all}}(K) = \|_{C\in D_{k_i}} P_{C\cup\Sigma_{k_i}}(K) = \|_{A_j\in A_{k_i}/_{\sim}} P_{A_j\cup\Sigma_{all}\cup\Sigma_{j}}(K)$, since $\Sigma_j=\Sigma_{k_i}$ and the synchronous product operation is idempotent. Together, we obtain that $K = \|_{i=1}^{m} \|_{A_j\in A_{k_i}/_{\sim}} P_{A_j\cup\Sigma_{all}\cup\Sigma_{j}}(K) = \|_{i=1}^{n} P_{A_i\cup\Sigma_{all}\cup\Sigma_{i}}(K)$. Finally, since $\Sigma_{all}\subseteq \Sigma_i$, for all $i$, $K$ is separable with respect to $(A_i\cup\Sigma_i)_{i=1}^{n}$.
  \end{proof}

  The procedure significantly depends on the computation and verification of conditional decomposability, which relies on the assumption that all shared events of the alphabets under (local) considerations are always included in $\Sigma_{all}$ (resp. $\Sigma_i$). This then allows us to use Theorem~\ref{thmgen} to check the property in polynomial time.

\section{Computation of Coobservable Sublanguages}\label{sec:comp}
  In this section, we discuss how to compute the languages $\R_i$ locally as sublanguages or superlanguages of $K_i$ so that they satisfy the sufficient conditions that make their synchronous product controllable, coobservable and included in $K$. 
  
  Consider the settings of decentralized control, and let $(\Sigma_i)_{i=1}^{n}$ be extensions of local alphabets $(A_{o,i})_{i=1}^{n}$ computed by the procedure {\sc Rcd} described in Section~\ref{sec:separability}, such that the specification $K$ is separable with respect to $(A_{o,i}\cup \Sigma_i)_{i=1}^{n}$. 
  We now apply results of modular control, where the status of an event is global. Namely, all shared events have the same status in all components where they appear. This is not in general the case in decentralized control. However, since every shared event appears in at least one $\Sigma_i$, the choice of $\Sigma_{uo,i}$ then ensures that the status of shared observable events is the same in all components where they appear. Recall that, for controllable events, we assume that if a supervisor observes a controllable event, then it can also control it. Formally, we assume that $A_{o,i} \cap A_c \subseteq A_{c,i}$, which must also hold after the extension, that is, $(A_{o,i}\cup \Sigma_{o,i}) \cap A_c \subseteq (A_{c,i} \cup \Sigma_{c,i})$, for $i=1,2,\ldots,n$. Therefore, we adapt the controllable status of events, if needed, to ensure this condition.
  
  Before we proceed, we summarize our assumptions and notation as Assumption~\ref{assumption1}. It allows us to keep the rest of the paper more concise by referring to Assumption~\ref{assumption1} rather than repeating the individual assumptions in the statements of theorems that follow.
  \begin{assumption}\label{assumption1}
    Let $(G,(A_{c,i})_{i=1}^{n},(A_{o,i})_{i=1}^{n})$ be a controlled generator over an alphabet $A$. Let $L=L(G)$, and let $K\subseteq L$ be a specification language over $A$. Let $(\Sigma_i)_{i=1}^{n}$ be extensions of local alphabets $(A_{o,i})_{i=1}^{n}$ computed by the procedure {\sc Rcd}, such that language $K$ is separable with respect to alphabets $(A_{o,i} \cup \Sigma_i)_{i=1}^{n}$, where the union of alphabets $A_{o,i} \cup \Sigma_i$ results in $A$, that is, $A = \bigcup_{i=1}^{n} (A_{o,i} \cup \Sigma_i)$. 

    We further assume that if a supervisor observes a controllable event, then it can also control it; namely, that $(A_{o,i}\cup \Sigma_{o,i}) \cap A_{c} \subseteq (A_{c,i} \cup \Sigma_{c,i})$, where $\Sigma_{o,i} = \Sigma_i \cap A_o$ are the controllable events of $\Sigma$ and $\Sigma_{uo,i} = \Sigma_i \setminus \Sigma_{o,i}$ are the uncontrollable events of $\Sigma$.

    For $i=1,2,\ldots,n$, let $P_{i+\Sigma_i}$ denote the projection from $A^*$ to $(A_{o,i}\cup\Sigma_i)^*$. Let $\R_i$ be languages that are controllable with respect to projection $P_{i+\Sigma_i}(L)$ of the plant language $L$ to alphabet $A_{o,i}\cup\Sigma_i$ and locally uncontrollable events $(A_{o,i}\cup \Sigma_i)_{uc} = (A_{o,i}\cup\Sigma_i)\cap A_{uc}$ such that their synchronous product $\R = \|_{i=1}^{n} \R_i$ is included in $K$. Furthermore, we assume that $\R_i$ are
    \begin{enumerate}
      \item either normal with respect to $P_{i+\Sigma_i}(L)$ and $A_{o,i}\cup\Sigma_{o,i}$,
      \item or observable with respect to $P_{i+\Sigma_i}(L)$ and $A_{o,i}\cup\Sigma_{o,i}$, and all controllable events are observable ($A_c \subseteq A_o$). \qed
    \end{enumerate}
  \end{assumption}
  
  Notice that normality implies observability~\cite{CL08}, hence every local language $\R_i$ is observable. However, if one of the local languages $\R_i$ is observable and not normal, then we require that all controllable events are observable.

\subsection{Main Result}
  We now state our main result showing how to use our framework to compute a controllable and coobservable sublanguage, and illustrate it on an example.
  
  An important feature of our computation is that we automatically obtain a coobservable sublanguage.
  \begin{theorem}\label{main}
    Consider Assumption~\ref{assumption1}. If the languages $\R_{i}$ are synchronously nonconflicting (in particular, if they are prefix-closed), then $\R = \|_{i=1}^{n} \R_{i}$ is a sublanguage of $K$ controllable with respect to $L$ and $A_{uc}$, and coobservable with respect to $L$ and $(A_{o,i}\cup\Sigma_{o,i})_{i=1}^{n}$.
  \end{theorem}
  \begin{proof}
    By definition, we have that $\R = \|_{i=1}^{n} \R_{i} \subseteq K$. By Lemmas~\ref{feng} and~\ref{fengpo} (Lemma~\ref{obsComposition2}) in the appendix, $\R$ is controllable and normal (observable) with respect to $\|_{i=1}^{n} P_{i+\Sigma_i}(L)$. Since $L \subseteq \|_{i=1}^{n} P_{i+\Sigma_i}(L)$, it is also controllable and normal (observable) with respect to $L$. Because $(A_{o,i}\cup\Sigma_{o,i}) \cap A_{c} \subseteq (A_{c,i}\cup\Sigma_{c,i})$ by the assumption and $\R_i$ are synchronously nonconflicting, that is, $\overline{\R} = \|_{i=1}^{n} \overline{\R_i}$, Theorem~\ref{newthm} implies that $\R = \|_{i=1}^{n} \R_{i}$ is coobservable with respect to $L$ and $(A_{o,i}\cup\Sigma_{o,i})_{i=1}^{n}$.
  \end{proof}

  We now illustrate our approach on a simple example.
  \begin{example}\label{example1}
    Consider the languages $K=\overline{\{aa,ba,bbd,abc\}}$ and $L=\overline{\{aac,abc,bac,bbd\}}$ over $A=\{a,b,c,d\}$, and alphabets $A_{o,1}=A_{c,1}=\{a,c\}$ and $A_{o,2}=A_{c,2}=\{b,d\}$. Then $K$ is not coobservable with respect to $L$ and $(A_{o,i})_{i=1}^{2}$, because none of the supervisors is able to distinguish between $ab$ and $ba$, where the continuation of $ba$ by $c$ within the plant leads outside the specification while   the continuation of $ab$ by $c$ remains within the specification. 

    We compute the extensions $\Sigma_1=\Sigma_2\supseteq A_{o,1}\cap A_{o,2}$ by the procedure {\sc Rcd} such that $K$ is separable with respect to $A_{o,1}\cup\Sigma_1$ and $A_{o,2}\cup\Sigma_2$. It is sufficient to take $\Sigma_1=\Sigma_2=\{b\}$, which needs to be communicated/observed by both supervisors. Since our system is with complete observations, we may compute
      $\R_1 = \overline{\{aa,abc,ba,bb\}}$ and
      $\R_2 = \overline{\{bbd\}}$ 
    as the supremal controllable sublanguages of $P_{i+\Sigma_i}(K)$ with respect to $P_{i+\Sigma_i}(L)$ and $(A_{o,i}\cup\Sigma_i)_{uc}$.
    By Theorem~\ref{main}, we have that $\R_1 \parallel \R_2$ is coobservable with respect to $L$ and the extended alphabets $\{a,b,c\}$ and $\{b,d\}$. Indeed, supervisor $S_1$ that exerts the control power over the event $c$ is now able to distinguish between the words $ab$, after which $c$ should be allowed, and $ba$, after which $c$ should be disabled.\qed
  \end{example}

\subsection{Construction of the Languages $\R_i$}\label{subsec:Ri}
  There are many ways how to compute the languages $\R_i$ discussed in the literature. In the case of full local observations, it is natural to define the language
  \begin{align}\label{supC}
      \R_{i} & = \supc_{i+\Sigma_i} = \supc(P_{i+\Sigma_i}(K), P_{i+\Sigma_i}(L), (A_{o,i}\cup\Sigma_i)_{uc})
  \end{align}
  as the supremal controllable sublanguage of $P_{i+\Sigma_i}(K)$ with respect to $P_{i+\Sigma_i}(L)$ and uncontrollable events $(A_{o,i}\cup\Sigma_i)_{uc}$. In the case of partial observations, we may define the language
  \begin{align*}
      \R_{i} & = \supcn(P_{i+\Sigma_i}(K), P_{i+\Sigma_i}(L), (A_{o,i}\cup\Sigma_i)_{uc}, A_{o,i}\cup\Sigma_{o,i})
  \end{align*}
  as the supremal controllable and normal sublanguage of $P_{i+\Sigma_i}(K)$ with respect to $P_{i+\Sigma_i}(L)$, $(A_{o,i}\cup\Sigma_i)_{uc}$ and $A_{o,i}\cup\Sigma_{o,i}$~\cite{CL08,brandt}. Similarly, if $A_c\subseteq A_o$, we can define $\R_i$ as the supremal controllable and relatively observable sublanguage~\cite{CaiZW15,KomendaMS14a}, or we can use any of the methods to compute a controllable and observable sublanguage discussed in the literature~\cite{FA199311,TakaiU03,YinL16}. 
  In these cases, we have that $\R_i \subseteq P_{i+\Sigma_i}(K)$, and separability of $K$ then implies that the synchronous product $\|_{i=1}^{n} \R_i$ is included in $K$ as required in Assumption~\ref{assumption1}. 
  
  However, we do not restrict language $\R_i$ to be included in $P_{i+\Sigma_i}(K)$. This allows us to define $\R_i$ in many different ways. For instance, we can define $\R_i$ as the infimal controllable (and normal/observable) superlanguage of $P_{i+\Sigma_i}(K)$ with respect to $P_{i+\Sigma_i}(L)$ and $(A_{o,i}\cup\Sigma_i)_{uc}$ (and $A_{o,i}\cup\Sigma_{o,i}$) as discussed in the literature~\cite{CL08,Kumar1995OSC,LafortuneChen1990,Rudie1990IPO}. We can further combine the approaches so that one of the $\R_i$ can be computed as a sublanguage and another one as a superlanguage, etc. In such a case, we do not get the assumption $\|_{i=1}^{n} \R_i \subseteq K$ by construction, but we need to check it. It is in general a PSPACE-complete problem.\footnote{There is a simple reduction from the finite-state automata intersection problem: Given a set of deterministic finite automata $\{G_i\}_{i=1}^{n}$ over a common alphabet $B$, is $\bigcap_{i=1}^{n} L_m(G_i)=\emptyset$? The problem is PSPACE-complete~\cite{Kozen77}. Let $A\cap B = \emptyset$. It is not hard to see that $\bigcap_{i=1}^{n} L_m(G_i)=\emptyset$ if and only if $\bigcap_{i=1}^{n} (L_m(G_i)\cup K) \subseteq K$. The reduction is polynomial since every $L_m(G_i)\cup K$ is represented by a generator computed from $G_i$ and the generator for $K$ in polynomial time by the standard product construction.} On the other hand, the advantage it brings is a potentially better (larger) solution as illustrated in Example~\ref{example4}.

  \begin{example}\label{example4}
    Let $L=\overline{\{ab,ba,bdau,dbau\}}$ be a language over $A=\{a,b,d,u\}$, and consider the alphabets $A_{o,1}=\{a,u\}$, $A_{c,1}=\{a,d\}$, $A_{o,2}=\{b,u\}$, $A_{c,2}=\{b\}$. Let $K=\overline{\{ab,ba,bd,db\}}$ be a specification. Then $K$ is not coobservable with respect to $L$ and $(A_{o,i})_{i=1}^{2}$ because, for $db\in K$, we have $a\in A_c$, $dba\in L\setminus K$, and $P_1(db)a=a\in K$ and $P_2(db)a=ba\in K$. Thus, none of the supervisors can disable $a$ after the word $db$.
    
    First, we compute only sublanguages. Let $\Sigma_1=\Sigma_2 = \{d,u\}$ be the extensions of local observations computed by {\sc Rcd}. Then $K$ is separable with respect to $(A_{o,i}\cup\Sigma_i)_{i=1}^{2}$. Notice that $P_{1+\Sigma_1}(K)=\{a,d,\eps\}$ and $P_{2+\Sigma_2}(K) = \overline{\{bd,db\}}$, and that $P_{1+\Sigma_1}(L) = \overline{\{a,dau\}}$ and $P_{2+\Sigma_2}(L) = \overline{\{bdu,dbu\}}$. Then $\R_1 = \supc_{1+\Sigma_1} = P_{1+\Sigma_1}(K)$ and $\R_2 = \supc_{2+\Sigma_2}=\{b,d,\eps\}$, and the solution $\R_1 \parallel \R_2 \subsetneq K$ gives us a strict subset of $K$.

    On the other hand, we can obtain the whole $K$ if we consider infimal controllable superlanguages of $P_{i+\Sigma_i}(K)$ instead of supremal controllable sublanguages. Then we obtain the infimal controllable superlanguages $\R_1 = P_{1+\Sigma}(K) = \supc_{1+\Sigma_1}$ and $\R_2 = P_{2+\Sigma}(K)$, and the solution $\R_1 \parallel \R_2 = K$ then gives us the whole specification as the resulting language.
  \qed\end{example}

  Another problem with infimal controllable superlanguages is that they do not exist for general languages, but only for prefix-closed languages. This issue can be avoided by the following choice of $\R_i$ based on the computation of prefix-closed superlanguages.

  \begin{lemma}
    Consider Assumption~\ref{assumption1}. Let $T_i$ be the prefix-closed infimal controllable (and normal/observable) superlanguage of $P_{i+\Sigma_i}(K)$ with respect to $P_{i+\Sigma_i}(L)$ and $(A_{o,i}\cup\Sigma_i)_{uc}$ (and $A_{o,i}\cup\Sigma_{o,i}$). Then $\R_{i} = P_{i+\Sigma_i}(K) \cup [T_{i} \setminus \overline{P_{i+\Sigma_i}(K)}]$ is controllable (and normal/observable) with respect to $P_{i+\Sigma_i}(L)$ and $(A_{o,i}\cup\Sigma_i)_{uc}$ (and $A_{o,i}\cup\Sigma_{o,i}$).
  \end{lemma}
  \begin{proof}
    We show that $\overline{\R_i} = T_i$. 
    Since $P_{i+\Sigma_i}(K)\subseteq T_{i}$ and $\R_{i} \subseteq T_{i}$, we have that $\overline{\R_{i}} \subseteq T_{i}$. To show that $T_{i}\subseteq \overline{\R_{i}}$, let $w\in T_{i}$. If $w\notin \overline{P_{i+\Sigma_i}(K)}$, then $w\in \R_{i}$ by definition. If $w\in \overline{P_{i+\Sigma_i}(K)}$, then there exists $v$ such that $wv\in P_{i+\Sigma_i}(K) \subseteq \R_{i}$, hence $w \in \overline{\R_{i}}$.
  \end{proof}

\subsection{Conditions for Optimality for Full Observations}
  In this subsection, we discuss conditions under which the solution of Theorem~\ref{main} is optimal in the sense of maximal permissiveness. That is, under which conditions the solution coincides with the supremal centralized supervisor, if it exists. Since no centralized optimal solution exists in the case of partial observations, we restrict our attention in this subsection only to the case of full observations, that is, we assume in this section that $A_{uo}=\emptyset$. However, we point out that a similar result to Proposition~\ref{optfo} can be obtain using the notions of {\em mutual normality}~\cite[Theorem~4.23]{KomendaS08} or {\em mutual observability}~\cite[Theorem~5.1]{Komenda200597} in the case the set $A_{uo}$ is nonempty. Theorem~\ref{optfoc} below can then be modified in the corresponding way.
  
  For $i=1,\ldots,n$, let a language $L_i$ over $A_i$ be prefix-closed. Languages $(L_i)_{i=1}^{n}$ are {\em mutually controllable\/} if $L_j(A_{i,u}\cap A_j)\cap P_jP_i^{-1}(L_i)\subseteq L_j$, for $i,j=1,2,\ldots,n$~\cite{LW02}. The following compatibility between supremal controllable sublanguages and the synchronous composition operator is known~\cite{LW02}. We state the result only for prefix-closed languages and refer the reader to Lee and Wong~\cite{LW02} for the conditions on general languages.

  \begin{proposition}\label{optfo}
    Assume that $A_{o,i}\cap A_{c}\subseteq A_{c,i}$. If the prefix-closed languages $L_i\subseteq A_i^*$ are mutually controllable, then
    $
      \|_{i=1}^n \supc (P_i(K),P_i(L),A_{i,u})= \supc (\|_{i=1}^n P_i(K),\|_{i=1}^n P_i(L), A_{uc})
    $
    holds true for any prefix-closed language $K\subseteq L$.
  \end{proposition}
  
  It may be that $P_{i+\Sigma_i}(L)$ and $P_{j+\Sigma_j}(L)$ are mutually controllable for fairly small alphabets $\Sigma_i$ and $\Sigma_j$. However, if we do not require that $K$ is separable, we cannot guarantee that the resulting supremal controllable sublanguage is included in $K$. This is the main issue with the approach in Komenda et al.~\cite{KMP06}. Therefore, in this paper, we compute the communications $(\Sigma_i)_{i=1}^{n}$ using the procedure {\sc Rcd}, such that $K$ is separable with respect to $(A_{o,i}\cup\Sigma_i)_{i=1}^{n}$, to ensure the inclusion of the resulting supremal controllable sublanguage in $K$. In addition, if $L$ is also separable with respect to $(A_{o,i}\cup\Sigma_i)_{i=1}^{n}$, which can again be ensured in the same way as for $K$, we obtain the optimal centralized solution.
  
  \begin{theorem}\label{optfoc}
    Consider Assumption~\ref{assumption1}. Let $K\subseteq L$ be prefix-closed languages, and let $K$ and $L$ be separable with respect to $(A_{o,i}\cup\Sigma_i)_{i=1}^{n}$. Let $\supc_{i+\Sigma_i}$ be defined by Equation~(\ref{supC}) above. If $P_{i+\Sigma_i}(L)$ and $P_{j+\Sigma_j}(L)$ are mutually controllable, for $i,j=1,2,\ldots,n$, then
    $
      \|_{i=1}^{n} \supc_{i+\Sigma_i} = \supc(K, L, A_{uc})
    $
    is coobservable with respect to $L$ and $(A_{o,i}\cup\Sigma_{i})_{i=1}^{n}$. 
  \end{theorem}
  \begin{proof}
    Since $\supc_{i+\Sigma_i}\subseteq P_{i+\Sigma_i}(K)$, $\|_{i=1}^{n} \supc_{i+\Sigma_i} \subseteq \|_{i=1}^{n} P_{i+\Sigma_i}(K) = K$ by separability of $K$. By Theorem~\ref{main}, $\|_{i=1}^{n} \supc_{i+\Sigma_i}$ is controllable with respect to $\|_{i=1}^{n}P_{i+\Sigma_i}(L) = L$ and $A_{uc}$, hence coobservable with respect to $L$ and $(A_{o,i}\cup\Sigma_{i})_{i=1}^{n}$ by Theorem~\ref{newthm}. Finally, by Proposition~\ref{optfo}, $\|_{i=1}^{n} \supc_{i+\Sigma_i} = \supc (K,L,A_{uc})$. 
  \end{proof}

  Whether mutual controllability can be fulfilled or not depends on the system. However, mutual controllability holds if all shared events are controllable. For instance, in Example~\ref{example1}, the only shared event between $A_{o,1} \cup \Sigma_1$ and $A_{o,2} \cup \Sigma_2$ is event $b$, which is controllable. Therefore, the languages $P_{1+\Sigma_1}(L)$ and $P_{2+\Sigma_2}(L)$ are mutually controllable, and Theorem~\ref{optfoc} implies that the parallel composition $\R_{1} \parallel \R_{2}$ coincides with the optimal monolithic solution $\supc(K,L,A_{uc})$, and is coobservable with respect to $L$ and the extended alphabets $\{a,b,c\}$ and $\{b,d\}$. 

  Depending on the system, mutual controllability may be a strong condition. We now present a result that ensures optimality and is based on the notions of an $L$-observer and local control consistency (LCC).
  
  A projection $P_k\colon A^* \to A_k^*$, for $A_k\subseteq A$, is an {\em $L$-observer\/} for a language $L\subseteq A^*$ if for all $s\in \overline{L}$, if $P_k(s)t\in P_k(L)$, then there exists $u\in A^*$ such that $su\in L$ and $P_k(u)=t$~\cite{wong98,pcl12}.
  The co-domain of a projection can always be extended to fulfill the condition. Although to compute the minimal extension is NP-hard, there is a polynomial-time algorithm to find an acceptable extension~\cite{FengWonham,pena2010}. The property also prevents the state explosion when computing projections. If $P$ is an $L$-observer, then the generator for $P(L)$ is not larger (usually much smaller) than the one for $L$.
  
  Let $L$ be a prefix-closed language over $A$, and let $A_k \subseteq A$. Projection $P_k\colon A^* \to A_k^*$ is {\em locally control consistent\/} (LCC) with respect to a word $s\in L$ if for all events $a_u\in A_k \cap A_{uc}$ such that $P_k(s) a_u\in P_k(L)$, it holds that either there does not exist any word $u\in (A \setminus A_k)^*$ such that $su a_u \in L$, or there exists a word $u \in (A_{uc} \setminus A_k)^*$ such that $su a_u \in L$. Projection $P_k$ is LCC with respect to $L$ if $P_k$ is LCC for all words of $L$~\cite{SB11,SB08}. Notice that LCC is a weaker condition than OCC defined in Zhong and Wonham~\cite{WZ91}.
  
  \begin{theorem}\label{thm_inc_opt_sol}
    Consider Assumption~\ref{assumption1}. If every projection $P_{i+\Sigma_i}$ is an $L$-observer and LCC for $L$, and the languages $\supc_{i+\Sigma_i}$ that are defined in (\ref{supC}) are synchronously nonconflicting (e.g., prefix-closed), then the language $\parallel_{i=1}^{n} \supc_{i+\Sigma_i} = \supc(K, L, A_{uc})$ is coobservable with respect to $L$ and $(A_{o,i}\cup\Sigma_{o,i})_{i=1}^{n}$.
  \end{theorem}
  \begin{proof}
    The identity $\|_{i = 1}^n \supc_{i+\Sigma_i} = \supc(K,L,A_{uc})$ has been shown in the literature~\cite{FLT,SB11}. Since $A_{uo} = \emptyset$, Theorem~\ref{newthm} finishes the proof.
  \end{proof}

  Note that both properties $L$-observer and LCC can be ensured by further extending the alphabets $\Sigma_i$ in polynomial time.
  
  A similar result to Theorem~\ref{thm_inc_opt_sol} can be obtained for supremal controllable and normal sublanguages. A condition under which the parallel composition of local supremal controllable and normal languages equals to the global supremal controllable and normal sublanguage can be found in~\cite[Theorem~25]{scl2011}.

\subsection{Conflicting Supervisors}
  So far, the theorems require that the local languages $\R_i$ are synchronously nonconflicting. The remaining question is thus the case of conflicting local supervisors. It is in general a PSPACE-complete problem to decide whether a parallel composition (of an unspecified number) of generators is nonblocking~\cite{rohloff}. However, Malik~\cite{Malikwodes2016} shows that current computers can explore more than 100 million of states using explicit algorithms without any optimization techniques. Moreover, it is possible to use an $L$-observer to alleviate the computational effort as used in the following construction.

  \begin{theorem}\label{thm22}
    Consider Assumption~\ref{assumption1}. Let $L_C \subseteq \|_{i=1}^{n} P_{\Sigma'}(\R_{i})$ be a language that is controllable and normal (observable) with respect to $P_{\Sigma'}(L)$, $\Sigma'_{uc}$ and $\Sigma'_{o}$, where $\Sigma' \subseteq A$ contains all events shared by any pair of $\R_i$ and $\R_j$, for $i\neq j$, and $P_{\Sigma'}\colon A^*\to \Sigma'^*$ is an $\R_{i}$-observer, for $i=1,\ldots,n$. If $(A_{o,i}\cup\Sigma_{o,i}\cup\Sigma'_{o}) \cap A_{c} \subseteq (A_{c,i} \cup \Sigma_{c,i}\cup \Sigma'_{c})$, then
    $
      \|_{i=1}^{n} (\R_{i} \parallel L_C)
    $
    is a sublanguage of $K$ controllable (and normal/observable) with respect to $L$ and $A_{uc}$ (and $A_o$) that is coobservable with respect to $(A_{o,i}\cup \Sigma_{o,i}\cup \Sigma'_{o})_{i=1}^{n}$, and whose components are synchronously nonconflicting.
  \end{theorem}
  \begin{proof}
    By definition and Lemma~\ref{lemma:Wonham}, we have that $L_C \subseteq \|_{i=1}^{n} P_{\Sigma'}(\R_i) = P_{\Sigma'}(\|_{i=1}^{n} \R_i) \subseteq P_{\Sigma'}(K)$. Thus, $\|_{i=1}^{n} (\R_{i} \parallel L_C) = (\|_{i=1}^{n} \R_{i}) \parallel L_C \subseteq K \parallel P_{\Sigma'}(K) = K$.
  
    To prove nonconflictness, we make use of Lemma~\ref{fengT41} in the appendix, which states that $\overline{\R_{i} \parallel L_C} = \overline{\R_{i}} \parallel \overline{L_C}$ if and only if $\overline{P_{\Sigma'}(\R_{i})\parallel L_C} = \overline{P_{\Sigma'}(\R_{i})} \parallel \overline{L_C}$. The second equation holds, because both sides are equal to $\overline{L_C}$. Using Lemma~\ref{fengT41} again, we have that $\overline{ \parallel_{i=1}^{n} (\R_{i} \parallel L_C) } = \|_{i=1}^{n} \overline{\R_{i} \parallel L_C}$ if and only if $\overline{\parallel_{i=1}^{n} P_{\Sigma'}(\R_{i} \parallel L_C)} = \|_{i=1}^{n} \overline{P_{\Sigma'}(\R_{i} \parallel L_C)}$. Lemma~\ref{fengT41} can be applied again, since $P_{\Sigma'}$ is an $(\R_{i} \parallel L_C)$-observer. It follows from Theorem~2 in Pena et al.~\cite{pcl06} saying that a composition of observers is an observer -- note that $P_{\Sigma'}$ is an $\R_{i}$-observer by assumption and an $L_C$-observer since it is an identity. Again, the latter equation holds, because $P_{\Sigma'}(\R_{i} \parallel L_C) = P_{\Sigma'}(\R_{i}) \parallel L_C = L_C$, by Lemma~\ref{lemma:Wonham} and the definition of $L_C$. Thus, both sides are equal to $\overline{L_C}$. To summarize,
    $
      \overline{\parallel_{i=1}^{n} (\R_{i}\parallel L_C)}
         = \|_{i=1}^{n} (\overline{\R_{i} \parallel L_C})
         = \|_{i=1}^{n} \overline{\R_{i}} \parallel \overline{L_C}.
    $

    To prove controllability (and normality/observability), note that $\R_{i}$ is controllable (and normal/observable) with respect to $P_{i+\Sigma_i}(L)$, and $L_C$ is controllable (and normal/observable) with respect to $P_{\Sigma'}(L)$. By Lemma~\ref{feng} (Lemmas~\ref{fengpo} and~\ref{obsComposition2}) in the appendix and the nonconflictness shown above, $\R_{i} \parallel L_C$ is controllable (and normal/observable) with respect to $P_{i+\Sigma_i}(L) \parallel P_{\Sigma'}(L)$ and $(A_{o,i}\cup\Sigma_{i}\cup\Sigma')\cap A_{uc}$. Similarly, $\|_{i=1}^{n} (\R_{i} \parallel L_C)$ is controllable (and normal/observable) with respect to $\|_{i=1}^{n} P_{i+\Sigma_i}(L) \parallel P_{\Sigma'}(L)$ and $A_{uc}$. Since $L\subseteq \|_{i=1}^{n} P_{i+\Sigma_i}(L)$ and $L\subseteq P_{\Sigma'}^{-1} P_{\Sigma'}(L)$, we have that $\|_{i=1}^{n} (\R_{i} \parallel L_C)$ is controllable (and normal/observable) with respect to $L$ and $A_{uc}$.
    
    Since $\R_i \parallel L_C$ is normal (observable) with respect to $P_{i+\Sigma_i}(L) \parallel P_{\Sigma'}(L)$, and it holds that $P_{i+\Sigma_{i}+\Sigma'}(L) \subseteq P_{i+\Sigma_i}(L) \parallel P_{\Sigma'}(L)$, we have that $\R_i \parallel L_C$ is normal (observable) with respect to $P_{i+\Sigma_{i}+\Sigma'}(L)$. Thus, using Theorem~\ref{newthm}, $\|_{i=1}^{n} (\R_{i} \parallel L_C)$ is coobservable with respect to $L$ and $(A_{o,i} \cup \Sigma_{o,i} \cup \Sigma'_o)_{i=1}^{n}$. In general, to apply Theorem~\ref{newthm}, we need to adjust the controllable status of events, if needed, to satisfy $(A_{o,i}\cup \Sigma_{o,i} \cup \Sigma'_{o}) \cap A_{c} \subseteq (A_{c,i} \cup \Sigma_{c,i}\cup \Sigma'_{c})$.
  \end{proof}

  In case of full observations, we strengthen the previous result by computing the language $L_C$ as a sublanguage of $\|_{i=1}^{n} P_{\Sigma'}(\R_i)$ controllable with respect to $\|_{i=1}^{n} \overline{P_{\Sigma'}(\R_i)}$ rather than to $P_{\Sigma'}(L)$, which may result in a larger language $L_C$ because $\|_{i=1}^{n} \overline{P_{\Sigma'}(\R_i)} \subseteq P_{\Sigma'}(L)$.
  \begin{theorem}
    Consider Assumption~\ref{assumption1}. Let $L_C$ be the supremal sublanguage of $\|_{i=1}^{n} P_{\Sigma'}(\R_{i})$ that is controllable with respect to $\|_{i=1}^{n} \overline{P_{\Sigma'}(\R_{i})}$ and $\Sigma'_{uc}$, where $\Sigma' \subseteq A$ contains all events shared by any pair of $\R_i$ and $\R_j$, for $i\neq j$, and $P_{\Sigma'}\colon A^*\to \Sigma'^*$ is an $\R_{i}$-observer, for $i=1,\ldots,n$. If every projection $P_{i+\Sigma_i}$ is an $L$-observer and LCC for $L$, and the projection $P_{\Sigma'}$ is LCC for $\|_{i=1}^{n} \overline{\R_{i}}$, then $\|_{i=1}^{n} (\R_{i} \parallel L_C) = \supc(K, L, A_{uc})$ is coobservable with respect to $L$ and $(A_{o,i}\cup \Sigma_{i} \cup \Sigma')_{i=1}^{n}$, whose components are synchronously nonconflicting. It is again under the assumption that $(A_{o,i}\cup \Sigma_i \cup \Sigma') \cap A_{c} \subseteq (A_{c,i} \cup \Sigma_{c,i} \cup \Sigma'_{c})$.
  \end{theorem}
  \begin{proof}
    Let $\supc=\supc(K,L,A_{uc})$. We first show that $\|_{i=1}^{n} (\R_{i} \parallel L_C)$ is a subset of $\supc$ coobservable with respect to $L$ and $(A_{o,i}\cup \Sigma_i \cup \Sigma')_{i=1}^{n}$, and its components are synchronously nonconflicting. 
    Similarly as in the proof of Theorem~\ref{thm22}, we can show that $\|_{i=1}^{n} (\R_{i} \parallel L_C) \subseteq K$ and that
    $
      \overline{\parallel_{i=1}^{n} (\R_{i}\parallel L_C)}
         = \|_{i=1}^{n} (\overline{\R_{i} \parallel L_C})
         = \|_{i=1}^{n} \overline{\R_{i}} \parallel \overline{L_C}.
    $

    To prove controllability of $\|_{i=1}^{n} (\R_{i} \parallel L_C)$ with respect to $L$ and $A_{uc}$, note that $\R_{i}$ is controllable with respect to $\overline{\R_{i}}$ and $(A_{o,i}\cup\Sigma_i)_{uc}$, and $L_C$ is controllable with respect to $\|_{j=1}^{n} \overline{P_{\Sigma'}(\R_{j})}$ and $\Sigma'_{uc}$. By Lemma~\ref{feng} in the appendix and the nonconflictness of $\R_i$ and $L_C$ shown above, $\R_{i} \parallel L_C$ is controllable with respect to $\overline{\R_{i}} \parallel (\|_{j=1}^{n} \overline{P_{\Sigma'}(\R_{j})})$ and $(A_{o,i}\cup\Sigma_i\cup\Sigma')_{uc}$. Using the same argument on $\R_i\parallel L_C$, for $i=1,\ldots,n$, we obtain that $\|_{i=1}^{n} (\R_{i} \parallel L_C)$ is controllable with respect to $\|_{i=1}^{n} (\overline{\R_{i}} \parallel \|_{j=1}^{n} \overline{P_{\Sigma'}(\R_{j})}) = \|_{i=1}^{n} \overline{\R_{i}}$ and $A_{uc}$. One more application of Lemma~\ref{feng} on $\overline{\R_i}$ gives that $\|_{i=1}^{n} \overline{\R_{i}}$ is controllable with respect to $\|_{i=1}^{n} P_{i+\Sigma_i}(L)$, hence with respect to $L\subseteq \|_{i=1}^{n} P_{i+\Sigma_i}(L)$, and $A_{uc}$. Using transitivity of controllability, Lemma~\ref{lem_trans} in the appendix, we obtain that $\|_{i=1}^{n} (\R_{i}\parallel L_C)$ is controllable with respect to $L$ and $A_{uc}$. 

    Since $\R_i \parallel L_C$ is trivially observable, Theorem~\ref{newthm} implies that $\|_{i=1}^{n} (\R_{i} \parallel L_C)$ is coobservable with respect to $L$ and $(A_{o,i}\cup\Sigma_{i}\cup \Sigma')_{i=1}^{n}$. It again holds under the assumption that $(A_{o,i}\cup \Sigma_i\cup \Sigma') \cap A_{c} \subseteq (A_{c,i} \cup \Sigma_{c,i} \cup \Sigma'_{c})$.

    It remains to show the other inclusion $\supc \subseteq \|_{i=1}^{n} (\R_{i} \parallel L_C)$. We show it in two steps. First we show that $\supc\subseteq \|_{i=1}^{n} \R_{i}$ and then that $P_{\Sigma'}(\supc) \subseteq L_C$.
    
    To show that $\supc\subseteq \|_{i=1}^{n} \R_{i}$, we prove that, for every $i$, $P_{i+\Sigma_i}(\supc) \subseteq \R_{i}$ by showing that $P_{i+\Sigma_i}(\supc)\subseteq P_{i+\Sigma_i}(K)$ is controllable with respect to $P_{i+\Sigma_i}(L)$. To this end, let $t \in \overline{P_{i+\Sigma_i}(\supc)}$, $a\in (A_{o,i}\cup\Sigma_i)\cap A_{uc}$, and $ta \in P_{i+\Sigma_i}(L)$. Then, there exists $s \in \overline{\supc}$ such that $P_{i+\Sigma_i}(s) = t$. Since $P_{i+\Sigma_i}$ is an $L$-observer, there exists $v \in A^*$ such that $sv \in L$ and $P_{i+\Sigma_i}(sv) = ta$, that is, $v = ua$ for some $u\in (A \setminus (A_{o,i}\cup\Sigma_i))^*$. The LCC property of $P_{i+\Sigma_i}$ for $L$ and $sua \in L$ imply that there exists $u' \in (A_{uc} \setminus (A_{o,i}\cup\Sigma_i))^*$ such that $su'a \in L$. Since $u'$ is uncontrollable, controllability of $\supc$ with respect to $L$ and $A_{uc}$ implies that $su'a \in \overline{\supc}$, that is, $P_{i+\Sigma_i}(su'a)=ta \in \overline{P_{i+\Sigma_i}(\supc)}$. Thus, $P_{i+\Sigma_i}(\supc)$ is controllable with respect to $P_{i+\Sigma_i}(L)$, hence $P_{i+\Sigma_i}(\supc)\subseteq \R_{i}$. 

    Finally, we show that $P_{\Sigma'}(\supc) \subseteq L_C$, i.e., that it is a subset of $\|_{i=1}^{n} P_{\Sigma'}(\R_{i})$ controllable with respect to $\|_{i=1}^{n} \overline{P_{\Sigma'}(\R_{i})}$ and $\Sigma'_{uc}$. Since $P_{i+\Sigma_i}(\supc) \subseteq \R_{i}$, we have that $\supc \subseteq \|_{i=1}^{n} P_{i+\Sigma_i}(\supc) \subseteq \|_{i=1}^{n} \R_{i}$. Applying projection $P_{\Sigma'}$, we obtain $P_{\Sigma'}(\supc) \subseteq P_{\Sigma'}(\|_{i=1}^{n} \R_{i}) = \|_{i=1}^{n} P_{\Sigma'}(\R_{i})$, where the last equality is by Lemma~\ref{lemma:Wonham}, see the appendix.
    Using Theorem~2 in Pena et al.~\cite{pcl06} saying that a composition of observers is an observer, the assumption on $P_{\Sigma'}$ to be an $\R_{i}$-observer, which also means that $P_{\Sigma'}$ is an $\overline{\R_{i}}$-observer, implies that $P_{\Sigma'}$ is a $(\|_{i=1}^{n} \overline{\R_{i}})$-observer. We show that $P_{\Sigma'}(\supc)$ is controllable with respect to $\|_{i=1}^{n} \overline{P_{\Sigma'}(\R_{i})}$ and $\Sigma_{uc}'$. To this end, let $t\in \overline{P_{\Sigma'}(\supc)}$, $a\in \Sigma'_{uc}$, and $ta\in \|_{i=1}^{n} \overline{P_{\Sigma'}(\R_{i})} = P_{\Sigma'}(\|_{i=1}^{n} \overline{\R_{i}})$. Then, there exists $s \in \overline{\supc}$ such that $P_{\Sigma'}(s) = t$. Since $P_{\Sigma'}$ is a $(\|_{i=1}^{n} \overline{\R_{i}})$-observer, there exists $v \in A^*$ such that $sv \in \|_{i=1}^{n} \overline{\R_{i}}$ and $P_{\Sigma'}(sv) = ta$, that is, $v = ua$ for some $u\in (A \setminus \Sigma')^*$. The LCC property of $P_{\Sigma'}$ for $\|_{i=1}^{n} \overline{\R_{i}}$ and $sua \in \|_{i=1}^{n} \overline{\R_{i}}$ then imply that there exists $u' \in (A_{uc} \setminus \Sigma')^*$ such that $su'a\in \|_{i=1}^{n} \overline{\R_{i}}$. Since $\|_{i=1}^{n} \overline{\R_{i}} \subseteq L$, and $u'$ is uncontrollable, controllability of $\supc$ with respect to $L$ and $A_{uc}$ implies that $su'a \in \overline{\supc}$. That is, $P_{\Sigma'}(su'a)=ta \in \overline{P_{\Sigma'}(\supc)}$, hence $P_{\Sigma'}(\supc)\subseteq L_C$.
  
    Together, $\supc = \|_{i=1}^{n} \R_{i} \parallel L_C = \|_{i=1}^{n} (\R_{i} \parallel L_C)$, which completes the proof.
  \end{proof}

  Notice that it is sufficient to require that $P_{\Sigma'}$ is an $L$-observer and LCC for $L$.
  
  In case of partial observations, one could obtain the global supremal controllable and normal sublanguage in a similar way under the assumptions similar to those discussed below Theorem~\ref{thm_inc_opt_sol}. Since, as already mentioned, there is no global optimal solution in the case of partial observations, we do not discuss this case in more detail.

\subsection{Complexity}\label{complexity}
  We now briefly discuss the complexity of our approach. Since we use standard notions, the complexity mainly depends on the complexity of corresponding algorithms for the computation of controllable and observable languages.
  
  In Assumption~\ref{assumption1}, we assume that $K$ is separable. If this is not the case, a polynomial-time algorithm~\cite{KMvS12cd} is used to find extensions $(\Sigma_{i})_{i=1}^{n}$ making the language separable with respect to extended alphabets. Then we compute local supervisors $\R_i$ using the standard algorithms discussed in Subsection~\ref{subsec:Ri}. This requires to compute the projections of $K$ and $L$, which is exponential in the worst case. However, the alphabet $\Sigma_i$ can be chosen so that the projection $P_{i+\Sigma_i}$ is $K$- and $L$-observer, which then results in the computation of those projections in polynomial time~\cite{wong98}. The computation of $\R_i$ is then performed in the respective time. It is polynomial in the case of full observations.
  To check/ensure optimality in the case of full observations, we either check whether the polynomially many pairs of languages are mutually controllable, which can be done in polynomial time, or further extend the alphabets $\Sigma_i$ in polynomial time so that the observer and LCC conditions of Theorem~\ref{thm_inc_opt_sol} are satisfied.
  Furthermore, Theorems~\ref{main} and~\ref{thm_inc_opt_sol} require that the computed supervisors are nonconflicting. This is a PSPACE-complete problem~\cite{rohloff}. However, this test can be skipped and we can directly compute the language $L_C$ from Theorem~\ref{thm22}, which may, in the worst case, require exponential space with respect to the number of local supervisors.

\section{Conclusion}
  In this paper, we have shown how to construct a solution to the decentralized control problem (a controllable and coobservable sublanguage of a specification) by using additional communications. Our approach relies on the notion of conditional decomposability recently studied by the authors, which overcomes the undecidable problem to find a separable sublanguage of the specification. The computation of local supervisors is fully decentralized and coobservability is guaranteed by construction.  We discussed two ways how to obtain the globally optimal solution in case of full observations if the computed languages are synchronously nonconflicting (prefix-closed). One is based on the notion of mutual controllability, the other on increasing the communication between supervisors. Indeed, both approaches can be combined as preferred. Our approach can be used for both prefix-closed and non-prefix-closed specifications. For conflicting supervisors, we showed how to impose nonconflictness, hence coobservability.

\appendix
\section{Proof of Theorem~\ref{thm6}}

  \newcounter{oldthm}
  \setcounter{oldthm}{\value{theorem}}
  \setcounter{theorem}{\value{thm6again}}
  \begin{theorem}
    The following problem is PSPACE-complete.
    \begin{itemize}
      \item[] \textsc{Input:} Alphabets $E_1,E_2,\ldots,E_n$ and a generator $H$ over $\cup_{i=1}^{n} E_i$.
      \item[] \textsc{Output:} {\tt Yes} if and only if $L_m(H)$ is separable with respect to $(E_i)_{i=1}^{n}$.
    \end{itemize}
  \end{theorem}
  \setcounter{theorem}{\value{oldthm}}
  \begin{proof}
    Standard techniques simulating a product automaton on-the-fly show that it belongs to PSPACE.
    To prove hardness, we reduce the finite-state automata intersection problem (INT): Given a set of deterministic finite automata $\{G_i\}_{i=1}^{n}$ over a common alphabet $\Sigma$. Is $\bigcap_{i=1}^{n} L_m(G_i)=\emptyset$? The problem is PSPACE-complete~\cite{Kozen77}.
    
    Let a set of deterministic automata $\{G_i\}_{i=1}^{n}$ with a common alphabet $\Sigma$ be an instance of INT. Without loss of generality, we assume that $n\ge 3$, since if $n$ is constant, then the problem is solvable in {\sc PTime}. 
    Let $G_i=(X_i,\Sigma,\delta_{i},x_{o}^{i},F_{i})$ and assume that all states of $G_i$ are reachable from the initial state $x_{o}^{i}$. This assumption does not change the complexity. We construct a deterministic automaton $H$ and alphabets $(E_i)_{i=1}^{n}$ in polynomial time such that $L(H)$ is separable if and only if $\bigcap_{i=1}^{n} L_m(G_i) = \emptyset$.

    To this end, we define the automaton $H = (X, E, \delta, q_0, X)$ so that the set of states is $X = \bigcup_{i=1}^{n} X_i \cup \{q_0,q_1,q_2,q_3\}$, where $q_0,q_1,q_2,q_3$ are new states, $E = \Gamma \cup \Sigma$, where $\Gamma=\{e_1,\ldots,e_n,c\}$ is an alphabet such that $\Gamma \cap \Sigma = \emptyset$, and the transition function is defined as follows. The initial state $q_0$ goes under $e_i$ to the initial state $x_{o}^{i}$ of $G_i$, $i=1,\ldots,n$, and for every $a$ in $\Sigma$, $q_0$ goes under $a$ to $q_3$. In $q_3$, there is a self-loop under every $a$ in $\Sigma$. The transitions inside every $G_i$ are unchanged. For every $e$ in $\{e_1,\ldots,e_n\}$ and every $i=1,\ldots,n$, we add a transition from $x_{o}^{i}$ to $q_1$ under $e$. State $q_1$ contains a self-loop for every $e$ in $\{e_1,\ldots,e_n\}$. Finally, for $i=1,\ldots,n$, we add a $c$-transition from all states of $F_i$ of $G_i$ to state $q_2$, cf. Fig.~\ref{figReduction}.
    
    To complete the reduction, we define $E_i = E \setminus \{e_i\}$, which defines the projection $P_i \colon E^* \to E_i^*$, $i=1,\ldots,n$. Note that the reduction is polynomial.
    \begin{figure}
      \centering
      \includegraphics[scale=.8]{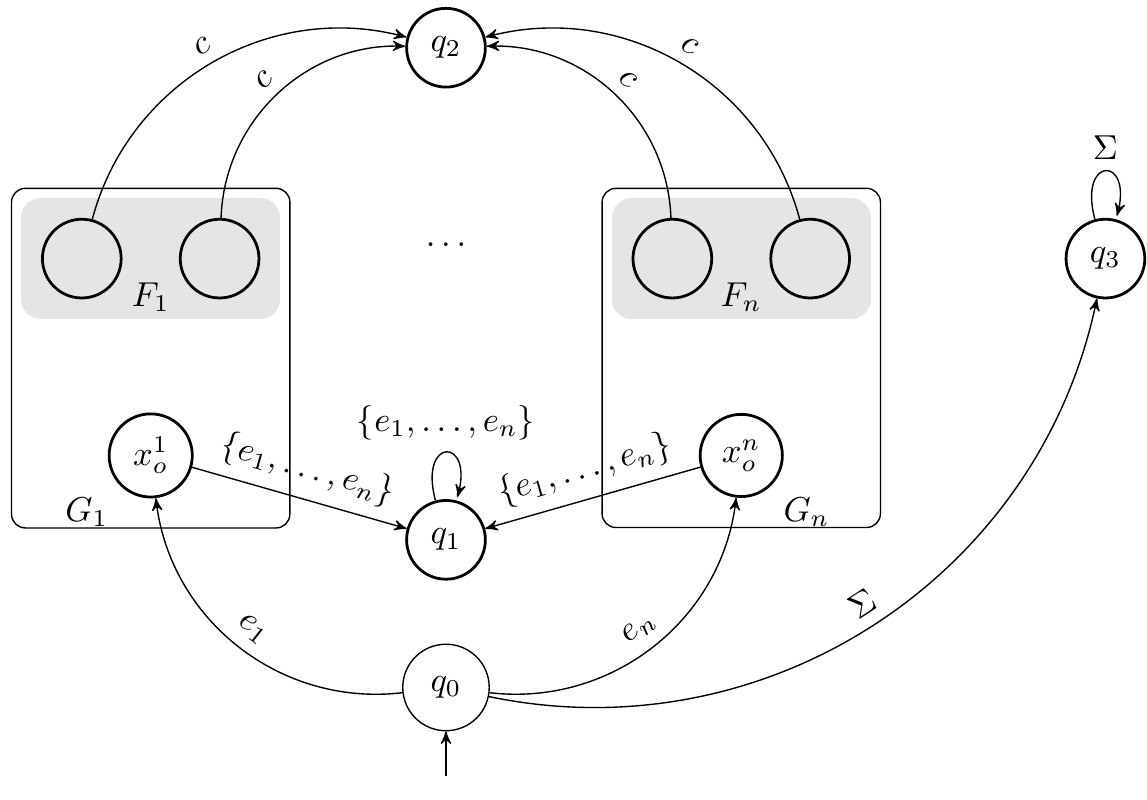}
      \caption{The automaton $H$}
      \label{figReduction}
    \end{figure}
    We show that $L(H)$ is separable with respect to $(E_i)_{i=1}^{n}$ if and only if $\bigcap_{i=1}^{n} L_m(G_i) = \emptyset$. 
    
    Assume that $t\in \bigcap_{i=1}^{n} L_m(G_i)$. Then $e_itc \in L(H)$, for all $i=1,\ldots,n$, which implies that $tc\in P_i^{-1}(P_i(e_itc))$, hence $tc \in \|_{i=1}^{n} P_i(L(H))$. However, $tc \notin L(H)$, which shows that $L(H)$ is not separable with respect to $(E_i)_{i=1}^{n}$.
    
    To prove the other direction, we assume that $L(H)$ is not separable and show that $\bigcap_{i=1}^{n} L_m(G_i) \neq\emptyset$. Let $w\in \|_{i=1}^{n} P_i(L(H))$ and $w\notin L(H)$. Note that $\Sigma^* \cup \{e_1,\ldots,e_n\}^*\subseteq L(H)$ because
    $
      L(H) = \Sigma^* \cup \{e_1,\ldots,e_n\}^* 
            \cup \bigcup_{i=1}^{n} e_iL_m(G_i)c 
            \cup \bigcup_{i=1}^{n} e_iL(G_i).
    $
    Therefore, we have that for $i=1,\ldots,n$, the word $w$ belongs to $P_i^{-1}P_i(L(H)) = (\Sigma\cup\{e_i\})^* \cup \{e_1,\ldots,e_n\}^* \cup P_i^{-1}(L_m(G_i)c) \cup \bigcup_{j\neq i} P_i^{-1}(e_jL_m(G_j)c) \cup P_i^{-1}(L(G_i)) \cup \bigcup_{j\neq i} P_i^{-1}(e_jL(G_j))$.

    We first show that if $c$ does not appear in $w$, then $w$ belongs to $L(H)$. In this case, based on the above observation, $w$ must contain at least one event from $\Sigma$ and at least one event from $\{e_1,\ldots,e_n\}$. Thus,
    \[
      w \in (\Sigma\cup\{e_i\})^* \cup \bigcup_{j\neq i} P_i^{-1}(e_jL(G_j))
    \]
    for $i=1,\ldots,n$, because $P_i^{-1}(L(G_i))\subseteq (\Sigma\cup\{e_i\})^*$. 
    Then $w\in \bigcap_{i=1}^{n} T_i$, where $T_i$ is one of the languages forming the union above. 
    
    If, for some $i\neq j$, $T_i=(\Sigma\cup\{e_i\})^*$ and $T_j=(\Sigma\cup\{e_j\})^*$, then $w\in\Sigma^*\subseteq L(H)$; a contradiction.
    
    If there is only one $i$ such that $T_i=(\Sigma\cup\{e_i\})^*$ and, for all $j\neq i$, $T_j = P_j^{-1}(e_{k_j}L(G_{k_j}))$, where $k_j\neq j$, then $w$ belongs to $\bigcap_{\ell=1}^{n} T_\ell$ if and only if $e_{k_j}=e_i$, for all $j\neq i$. Hence, for $j\neq i$, $T_i\cap T_j = e_iL(G_i)$, which implies that $\bigcap_{\ell=1}^{n} T_\ell = e_iL(G_i) \subseteq L(H)$; a contradiction.
    
    The last option is that, for all $i$, $T_i=P_i^{-1}(e_{j_i}L(G_{j_i}))$, where $j_i\neq i$. Then there exist $j_k\neq j_\ell$ such that $e_{j_k}\neq e_{j_\ell}$. Without loss of generality, we assume that $w=v_1 e_{j_k} v_2 e_{j_\ell} w'$, where $v_1v_2w' \in (\Sigma\cup\{e_1,\ldots,e_n\})^*$. Since $w\in T_k=P_k^{-1}(e_{j_k}L(G_{j_k}))$, $P_k(w)\in e_{j_k}L(G_{j_k})$, hence $P_k(v_1)=\eps$ and $P_k(v_2 e_{j_\ell}w')\in\Sigma^*$, which implies that $j_\ell = k$, $v_1\in\{e_k\}^*$ and $v_2w'\in (\Sigma\cup\{e_k\})^*$. Similarly, $P_\ell(w)\in e_{j_\ell}L(G_{j_\ell})$ implies that $j_k = \ell$, $v_1v_2\in\{e_\ell\}^*$ and $w'\in (\Sigma\cup\{e_\ell\})^*$. Together, $v_1v_2=\eps$, $w'\in\Sigma^*$, and $w=e_\ell e_k w'$, for $k\neq\ell$. By the assumption, there is a projection $P_m$ such that $P_m \notin \{P_k,P_\ell\}$. Since $w\in T_m$, $P_m(w)\in (\Sigma \cup \{e_{j_m}\})^*$, and $P_m(e_\ell)=\eps$ or $P_m(e_k)=\eps$. The first case gives that $P_m=P_\ell$, the second that $P_m=P_k$, which is a contradiction.
    
    Thus, we have show that if $c$ does not appear in $w$, then $w$ belongs to $L(H)$.
    
    Assume that $c$ appears in $w$. By the analysis above, $w\in \bigcap_{i=1}^{n} [ P_i^{-1}(L_m(G_i)c) \cup \bigcup_{j\neq i} P_i^{-1}(e_jL_m(G_j)c) ]$. It implies that there is exactly one $c$ in $w$. Again, $w\in \bigcap_{i=1}^{n} T_i$, where $T_i$ is one of the elements of the union.

    Analogously as above, if, for all $i$, $T_i = P_i^{-1}(e_{j_i}L_m(G_{j_i})c)$, where $j_i\neq i$, then there are $j_k\neq j_\ell$ such that $e_{j_k}\neq e_{j_\ell}$. Without loss of generality, let $w=v_1 e_{j_k} v_2 e_{j_\ell} w' c w''$, where $v_1v_2w'w'' \in (\Sigma\cup\{e_1,\ldots,e_n\})^*$. Then $P_k(w)\in e_{j_k}L_m(G_{j_k})c$, hence we have that $P_k(v_1)=\eps$ and $P_k(v_2 e_{j_\ell}w'w'')\in\Sigma^*$, which implies that $j_\ell = k$, $v_1\in\{e_k\}^*$ and $v_2w'w''\in (\Sigma\cup\{e_k\})^*$. Similarly, $P_\ell(w)\in e_{j_\ell}L_m(G_{j_\ell})c$ implies that $j_k = \ell$, $v_1v_2\in\{e_\ell\}^*$ and $w'w''\in (\Sigma\cup\{e_\ell\})^*$. Together, $v_1v_2=\eps$, $w'w''\in\Sigma^*$, and $w=e_\ell e_k w'cw''$, for $k\neq\ell$. Let $P_m\notin\{P_k,P_\ell\}$ be a projection. Since $P_m(w)\in (\Sigma \cup \{e_{j_m},c\})^*$, $P_m(e_\ell)=\eps$ or $P_m(e_k)=\eps$. The first case gives that $P_m=P_\ell$, the second that $P_m=P_k$, which is a contradiction.

    Thus, there must exist $i$ such that $T_i = P_i^{-1}(L_m(G_i)c)$. Then $P_i(w)\in L_m(G_i)c$, which implies that $w\in(\Sigma\cup\{e_i,c\})^*$. This means that, for $j\neq i$, $P_j(w) = w \in L_m(G_j)c \cup e_iL_m(G_i)c$. 
    If $T_j = e_iL_m(G_i)c$, for some $j\neq i$, then $w\in (\Sigma\cup\{e_i,c\})^* \cap e_iL_m(G_i)c = e_iL_m(G_i)c \subseteq L(H)$; a contradiction again. 
    Thus, it must be that for every $j\neq i$, $T_j = L_m(G_j)c$. Then $w\in \bigcap_{i=1}^{n} T_i = \bigcap_{i=1}^{n} L_m(G_i)c$ implies that $\bigcap_{i=1}^{n} L_m(G_i)\neq\emptyset$.
  \end{proof}

\section{Auxiliary Results}
  \begin{lemma}[\cite{FLT}]\label{feng}
    For $i=1,2$, let $L_i$ be a prefix-closed language over $A_i$, and let $K_i \subseteq L_i$ be controllable with respect to $L_i$ and $A_{i,uc}$. Let $A = A_1\cup A_2$. If $K_1$ and $K_2$ are synchronously nonconflicting, then $K_1\parallel K_2$ is controllable with respect to $L_1\parallel L_2$ and $A_{uc}$.
  \end{lemma}

  \begin{lemma}\label{fengpo}
    For $i=1,2$, let $L_i$ be a prefix-closed language over $A_i$, and let $K_i \subseteq L_i$ be normal with respect to $L_i$, $A_{i,uc}$ and $P_{o,i} \colon A_i^* \to A_{o,i}^*$. Let $A = A_1 \cup A_2$. If $K_1$ and $K_2$ are synchronously nonconflicting, then $K_1\parallel K_2$ is normal with respect to $L_1\parallel L_2$, $A_{uc}$ and $P_o\colon A^* \to A_o^*$.
  \end{lemma}
  \begin{proof}
    $P_o^{-1}P_o(\overline{K_1\parallel K_2}) \cap L_1\parallel L_2 \subseteq P_{o,1}^{-1}P_{o,1}(\overline{K_1}) \parallel P_{o,2}^{-1}P_{o,2}(\overline{K_2}) \parallel L_1\parallel L_2 = \overline{K_1} \parallel \overline{K_2} = \overline{K_1\parallel K_2}$. As the other inclusion always holds, the proof is complete.
  \end{proof}
  
  \begin{lemma}\label{obsComposition2}
    For $i=1,2$, let $L_i$ be a prefix-closed language over $A_i$, and let $K_i \subseteq L_i$ be observable with respect to $L_i$, $A_{i,uc}$ and $P_{o,i}\colon A_i^* \to A_{o,i}^*$. Let $A = A_1 \cup A_2$. If $K_1$ and $K_2$ are synchronously nonconflicting, then $K_1\parallel K_2$ is observable with respect to $L_1\parallel L_2$, $A_{uc}$ and $P_o\colon A^* \to A_{o}^*$.
  \end{lemma}
  \begin{proof}
    Let $s, s'\in A^*$ be such that $P_o(s) = P_o(s')$. Let $a \in A$ and assume that $sa \in \overline{K_1\parallel K_2}$, $s' \in \overline{K_1\parallel K_2}$, and $s'a \in L_1\parallel L_2$. Let $P_{i}\colon A^* \to A_i^*$, for $i=1,2$. Then $P_i(sa)\in \overline{K_i}$, $P_i(s')\in \overline{K_i}$, and $P_i(s'a)\in L_i$ imply that $P_i(s'a)\in \overline{K_i}$, by observability of $K_i$ with respect to $L_i$. Thus, $s'a\in \overline{K_1}\parallel \overline{K_2} = \overline{K_1 \parallel K_2}$.
  \end{proof}

  \begin{lemma}[\cite{pcl06}]\label{fengT41}
    Let $L_i\subseteq A_i^*$, $i\in J$, and $\bigcup_{k,\ell\in J}^{k\neq\ell} (A_k\cap A_\ell)\subseteq A_0$. If $P_{i,0}\colon A_i^* \to (A_i\cap A_0)^*$ is an $L_i$-observer, for $i\in J$, then $\overline{\|_{i\in J} L_i} = \|_{i\in J} \overline{L_i}$ if and only if $\overline{\|_{i\in J} P_{i,0}(L_i)} = \|_{i\in J} \overline{P_{i,0}(L_i)}$. 
  \end{lemma}
  
  \begin{lemma}[\cite{Won12}]\label{lemma:Wonham}
    Let $P_k \colon A^* \to A_k^*$ be a projection, and let $L_i\subseteq A_i^*$, where $A_i \subseteq A$, for $i=1,2$, and $A_1\cap A_2 \subseteq A_k$. Then $P_k(L_1\| L_2)=P_k(L_1) \| P_k(L_2)$.
  \end{lemma}
  
  \begin{lemma}[\cite{KMvS12}]\label{lem_trans}
    Let $K \subseteq L \subseteq M$ be such that $K$ is controllable with respect to $\overline{L}$ and $L$ is controllable with respect to $\overline{M}$. Then $K$ is controllable with respect to $\overline{M}$.
  \end{lemma}

\subsection*{Acknowledgements}
  The authors are grateful to the anonymous reviewers, whose comments and suggestions significantly improved the paper.
  The work was supported by RVO 67985840, by GA{\v C}R in project GA15-02532S and by the German Research Foundation (DFG) in Emmy Noether grant KR~4381/1-1 (DIAMOND).

\bibliographystyle{plain}
\bibliography{most_extension}

\begin{thebibliography}{10}

\bibitem{BL00}
G.~Barrett and S.~Lafortune.
\newblock Decentralized supervisory control with communicating controllers.
\newblock {\em IEEE Transactions on Automatic Control}, 45(9):1620--1638, 2000.

\bibitem{brandt}
R.~D. Brandt, V.~Garg, R.~Kumar, F.~Lin, S.~I. Marcus, and W.~M. Wonham.
\newblock Formulas for calculating supremal controllable and normal
  sublanguages.
\newblock {\em Systems \& Control Letters}, 15(2):111--117, 1990.

\bibitem{pcl12}
H.J. Bravo, A.E.C. {Da Cunha}, P.N. Pena, R.~Malik, and J.E.R. Cury.
\newblock Generalised verification of the observer property in discrete event
  systems.
\newblock In {\em WODES}, pages 337--342, Mexico, 2012.

\bibitem{CaiZW15}
Kai Cai, Renyuan Zhang, and W.~Murray Wonham.
\newblock Relative observability of discrete-event systems and its supremal
  sublanguages.
\newblock {\em IEEE Transactions on Automatic Control}, 60(3):659--670, 2015.

\bibitem{CL08}
C.G. Cassandras and S.~Lafortune.
\newblock {\em Introduction to discrete event systems}.
\newblock Springer, second edition, 2008.

\bibitem{CK11}
H.~Chakib and A.~Khoumsi.
\newblock Multi-decision supervisory control: Parallel decentralized
  architectures cooperating for controlling discrete event systems.
\newblock {\em IEEE Transactions on Automatic Control}, 56(11):2608--2622,
  2011.

\bibitem{FA199311}
Jinghuai Fa, Xiaojun Yang, and Yingping Zheng.
\newblock Formulas for a class of controllable and observable sublanguages
  larger than the supremal controllable and normal sublanguage.
\newblock {\em Systems \& Control Letters}, 20(1):11--18, 1993.

\bibitem{FLT}
L.~Feng.
\newblock {\em Computationally Efficient Supervisor Design for Discrete-Event
  Systems}.
\newblock PhD thesis, University of Toronto, Ontario, Canada, 2007.

\bibitem{FengWonham}
L.~Feng and W.M. Wonham.
\newblock On the computation of natural observers in discrete-event systems.
\newblock {\em Discrete Event Dynamic Systems}, 20(1):63--102, 2010.

\bibitem{JKumar00}
S.~Jiang and R.~Kumar.
\newblock Decentralized control of discrete event systems with specializations
  to local control and concurrent systems.
\newblock {\em IEEE Transactions on Systems, Man, and Cybernetics---Part B:
  Cybernetics}, 30(5):653--660, 2000.

\bibitem{JiangKG03}
S.~Jiang, R.~Kumar, and H.~E. Garcia.
\newblock Optimal sensor selection for discrete-event systems with partial
  observation.
\newblock {\em IEEE Transactions on Automatic Control}, 48(3):369--381, 2003.

\bibitem{KMP06}
J.~Komenda, H.~Marchand, and S.~Pinchinat.
\newblock A constructive and modular approach to decentralized supervisory
  control problems.
\newblock {\em IFAC Proceedings Volumes}, 39(17):111--116, 2006.
\newblock 3rd {IFAC} Workshop on Discrete-Event System Design.

\bibitem{allerton2013}
J.~Komenda and T.~Masopust.
\newblock A bridge between decentralized and coordination control.
\newblock In {\em Allerton Conference on Communication, Control, and
  Computing}, pages 966--972, 2013.

\bibitem{scl2011}
J.~Komenda, T.~Masopust, and J.H. van Schuppen.
\newblock Synthesis of controllable and normal sublanguages for discrete-event
  systems using a coordinator.
\newblock {\em Systems \& Control Letters}, 60(7):492--502, 2011.

\bibitem{KMvS12cd}
J.~Komenda, T.~Masopust, and J.H. van Schuppen.
\newblock On conditional decomposability.
\newblock {\em Systems \& Control Letters}, 61(12):1260--1268, 2012.

\bibitem{KMvS12}
J.~Komenda, T.~Masopust, and J.H. van Schuppen.
\newblock Supervisory control synthesis of discrete-event systems using a
  coordination scheme.
\newblock {\em Automatica}, 48(2):247--254, 2012.

\bibitem{KMvS13JDEDS}
J.~Komenda, T.~Masopust, and J.H. van Schuppen.
\newblock Coordination control of discrete-event systems revisited.
\newblock {\em Discrete Event Dynamic Systems}, 25(1-2):65--94, 2015.

\bibitem{KomendaMS14a}
J.~Komenda, T.~Masopust, and J.H. van Schuppen.
\newblock Relative observability in coordination control.
\newblock In {\em {IEEE} International Conference on Automation Science and
  Engineering ({CASE})}, pages 75--80, 2015.

\bibitem{Komenda200597}
Jan Komenda and Jan~H. van Schuppen.
\newblock Modular antipermissive control of discrete-event systems.
\newblock {\em {IFAC} World Congress}, 38(1):97--102, 2005.

\bibitem{KomendaS08}
Jan Komenda and Jan~H. van Schuppen.
\newblock Modular control of discrete-event systems with coalgebra.
\newblock {\em IEEE Transactions on Automatic Control}, 53(2):447--460, 2008.

\bibitem{KozakW95}
P.~Koz{\' a}k and W.~M. Wonham.
\newblock Fully decentralized solutions of supervisory control problems.
\newblock {\em IEEE Transactions on Automatic Control}, 40(12):2094--2097,
  1995.

\bibitem{Kozen77}
D.~Kozen.
\newblock Lower bounds for natural proof systems.
\newblock In {\em FOCS}, pages 254--266, 1977.

\bibitem{KT07}
R.~Kumar and S.~Takai.
\newblock Inference-based ambiguity management in decentralized
  decision-making: Decentralized control of discrete event systems.
\newblock {\em IEEE Transactions on Automatic Control}, 52(10):1783--1794,
  2007.

\bibitem{Kumar1995OSC}
Ratnesh Kumar and Vijay~K. Garg.
\newblock Optimal supervisory control of discrete event dynamical systems.
\newblock {\em SIAM Journal on Control and Optimization}, 33(2):419--439, 1995.

\bibitem{LafortuneChen1990}
S.~Lafortune and E.~Chen.
\newblock The infimal closed controllable superlanguage and its application in
  supervisory control.
\newblock {\em IEEE Transactions on Automatic Control}, 35(4):398--405, 1990.

\bibitem{LW02}
S.-H. Lee and K.C. Wong.
\newblock Structural decentralized control of concurrent discrete-event
  systems.
\newblock {\em European Journal of Control}, 8(5):477--491, 2002.

\bibitem{LinSSWS14}
L.~Lin, A.~Stefanescu, R.~Su, W.~Wang, and A.R. Shehabinia.
\newblock Towards decentralized synthesis: Decomposable sublanguage and joint
  observability problems.
\newblock In {\em American Control Conference}, pages 2047--2052, 2014.

\bibitem{Malikwodes2016}
R.~Malik.
\newblock Programming a fast explicit conflict checker.
\newblock In {\em WODES}, pages 438--443, 2016.

\bibitem{pcl06}
P.N. Pena, J.E.R. Cury, and S.~Lafortune.
\newblock Verification of nonconflict of supervisors using abstractions.
\newblock {\em IEEE Transactions on Automatic Control}, 54(12):2803--2815,
  2009.

\bibitem{pena2010}
P.N. Pena, J.E.R. Cury, R.~Malik, and S.~Lafortune.
\newblock Efficient computation of observer projections using {OP}-verifiers.
\newblock In {\em WODES}, pages 416--421, 2010.

\bibitem{RW87}
P.J. Ramadge and W.M. Wonham.
\newblock Supervisory control of a class of discrete event processes.
\newblock {\em SIAM Journal on Control and Optimization}, 25(1):206--230, 1987.

\bibitem{RickerR07}
S.~L. Ricker and K.~Rudie.
\newblock Knowledge is a terrible thing to waste: Using inference in
  discrete-event control problems.
\newblock {\em IEEE Transactions on Automatic Control}, 52(3):428--441, 2007.

\bibitem{RR00}
S.L. Ricker and K.~Rudie.
\newblock Know means no: Incorporating knowledge into discrete-event control
  systems.
\newblock {\em IEEE Transactions on Automatic Control}, 45(9):1656--1668, 2000.

\bibitem{RohloffKK06}
K.~Rohloff, S.~Khuller, and G.~Kortsarz.
\newblock Approximating the minimal sensor selection for supervisory control.
\newblock {\em Discrete Event Dynamic Systems}, 16(1):143--170, 2006.

\bibitem{rohloff}
K.~Rohloff and S.~Lafortune.
\newblock {PSPACE}-completeness of modular su\-per\-vi\-so\-ry control
  problems.
\newblock {\em Discrete Event Dynamic Systems}, 15:145--167, 2005.

\bibitem{RohloffYL03}
K.~Rohloff, T.{-}S. Yoo, and S.~Lafortune.
\newblock Deciding co-observability is {PSPACE}-complete.
\newblock {\em IEEE Transactions on Automatic Control}, 48(11):1995--1999,
  2003.

\bibitem{RudieLL03}
K.~Rudie, S.~Lafortune, and F.~Lin.
\newblock Minimal communication in a distributed discrete-event system.
\newblock {\em IEEE Transactions on Automatic Control}, 48(6):957--975, 2003.

\bibitem{RudieW90}
K.~Rudie and W.~M. Wonham.
\newblock Supervisory control of communicating processes.
\newblock In {\em Protocol Specification, Testing and Verification X}, pages
  243--257, 1990.

\bibitem{RW92}
K.~Rudie and W.M. Wonham.
\newblock Think globally, act locally: Decentralized supervisory control.
\newblock {\em IEEE Transactions on Automatic Control}, 37(11):1692--1708,
  1992.

\bibitem{Rudie1990IPO}
Karen Rudie and W.~Murray Wonham.
\newblock The infimal prefix-closed and observable superlanguage of given
  language.
\newblock {\em Systems \& Control Letters}, 15(5):361--371, 1990.

\bibitem{SB08}
K.~Schmidt and C.~Breindl.
\newblock On maximal permissiveness of hierarchical and modular supervisory
  control approaches for discrete event systems.
\newblock In {\em WODES}, pages 462--467, 2008.

\bibitem{SB11}
K.~Schmidt and C.~Breindl.
\newblock Maximally permissive hierarchical control of decentralized discrete
  event systems.
\newblock {\em IEEE Transactions on Automatic Control}, 56(4):723--737, 2011.

\bibitem{Takai98}
S.~Takai.
\newblock On the language generated under fully decentralized supervision.
\newblock {\em IEEE Transactions on Automatic Control}, 43(9):1253--1256, 1998.

\bibitem{TKU08}
S.~Takai and R.~Kumar.
\newblock Synthesis of inference-based decentralized control for discrete event
  systems.
\newblock {\em IEEE Transactions on Automatic Control}, 53(2):522--534, 2008.

\bibitem{TakaiU03}
Shigemasa Takai and Toshimitsu Ushio.
\newblock Effective computation of an {$L_{m}(G)$}-closed, controllable, and
  observable sublanguage arising in supervisory control.
\newblock {\em Systems \& Control Letters}, 49(3):191--200, 2003.

\bibitem{Th05}
J.G. Thistle.
\newblock Undecidability in decentralized supervision.
\newblock {\em Systems \& Control Letters}, 54(5):503--509, 2005.

\bibitem{Tr04}
S.~Tripakis.
\newblock Undecidable problems of decentralized observation and control on
  regular languages.
\newblock {\em Information Processing Letters}, 90(1):21--28, 2004.

\bibitem{WangGLL11}
W.~Wang, A.~R. Girard, S.~Lafortune, and F.~Lin.
\newblock On codiagnosability and coobservability with dynamic observations.
\newblock {\em IEEE Transactions on Automatic Control}, 56(7):1551--1566, 2011.

\bibitem{WangLL08}
W.~Wang, S.~Lafortune, and F.~Lin.
\newblock Optimal sensor activation in controlled discrete event systems.
\newblock In {\em Conference on Decision and Control}, pages 877--882. {IEEE},
  2008.

\bibitem{WH1991}
Y.~Willner and M.~Heymann.
\newblock Supervisory control of concurrent discrete-event systems.
\newblock {\em International Journal of Control}, 54(5):1143--1169, 1991.

\bibitem{wong98}
K.~Wong.
\newblock On the complexity of projections of discrete-event systems.
\newblock In {\em WODES}, pages 201--206, 1998.

\bibitem{Won12}
W.M. Wonham.
\newblock Supervisory control of discrete-event systems. {U}niversity of
  {T}oronto, 2012.
\newblock Available at http://www.control.utoronto.ca/DES/.

\bibitem{YinL16}
X.~Yin and S.~Lafortune.
\newblock Synthesis of maximally permissive supervisors for partially-observed
  discrete-event systems.
\newblock {\em IEEE Transactions on Automatic Control}, 61(5):1239--1254, 2016.

\bibitem{YLL02}
T.S. Yoo and S.~Lafortune.
\newblock A general architecture for decentralized supervisory control of
  discrete-event systems.
\newblock {\em Discrete Event Dynamic Systems}, 12(3):335--377, 2002.

\bibitem{YL04}
T.S. Yoo and S.~Lafortune.
\newblock Decentralized supervisory control with conditional decisions:
  Supervisor existence.
\newblock {\em IEEE Transactions on Automatic Control}, 49(11):1886--1904,
  2004.

\bibitem{WZ91}
H.~Zhong and W.~M. Wonham.
\newblock On the consistency of hierarchical supervision in discrete-event
  systems.
\newblock {\em IEEE Transactions on Automatic Control}, 35(10):1125--1134,
  1990.

\end{thebibliography}

\end{document}